\documentclass[preprint,12pt]{elsarticle}

\usepackage{amsmath,amssymb,amsfonts} 
\usepackage{amsthm}                  
\usepackage{graphicx}                
\usepackage{booktabs}                
\usepackage{xcolor}                  
\usepackage{geometry}                
\usepackage{lipsum}                  
\usepackage{algorithm}
\usepackage{algorithmicx}
\usepackage{algpseudocode}

\usepackage{float}
\usepackage{multirow}
\usepackage{mathrsfs}
\usepackage[title]{appendix}
\usepackage{textcomp}
\usepackage{manyfoot}
\usepackage{listings}

\usepackage[unicode]{hyperref}
\hypersetup{
    colorlinks=true,
    linkcolor=blue,
    filecolor=magenta,      
    urlcolor=cyan,
    pdftitle={Memory-Type Null Controllability of Heat Equations with Delay Effects},
    pdfauthor={Dev Prakash Jha, Raju K. George},
    pdfsubject={Controllability of Evolution Equations},
    pdfkeywords={Evolution equation, Carleman estimates, null controllability},
}
\usepackage{cleveref}

\theoremstyle{plain}
\newtheorem{theorem}{Theorem}[section]
\newtheorem{proposition}[theorem]{Proposition}
\newtheorem{corollary}[theorem]{Corollary}
\newtheorem{lemma}[theorem]{Lemma}
\newtheorem{assumption}[theorem]{Assumption}

\theoremstyle{remark}

\theoremstyle{definition}

\newtheorem{remark}{Remark}[section]
\newtheorem{definition}{Definition}[section]

\begin{document}

\begin{frontmatter}

\title{Observability and Null Controllability for Parabolic Systems with Delay and Memory}

\author[inst1]{\texorpdfstring{Dev Prakash Jha\corref{cor1}}{Dev Prakash Jha}}
\ead{devprakash.22@res.iist.ac.in}

\author[inst1]{Raju K. George}
\ead{george@iist.ac.in}

\cortext[cor1]{Corresponding author}

\address[inst1]{Department of Mathematics, Indian Institute of Space Science and Technology, Valiamala P.O., Thiruvananthapuram 695547, Kerala, India}
\begin{abstract}
This work develops a controllability framework for parabolic evolution equations
		with simultaneous time-delay and memory effects. We introduce the notion of
		\emph{delay and memory-type null controllability}, which strengthens classical
		null controllability by requiring the terminal state, the delayed contribution,
		and the prescribed terminal memory functional to vanish. We establish an
		abstract duality principle connecting this property with an 
		observability inequality for the corresponding adjoint system on an extended
		time interval. Under suitable unique continuation and well-posedness
		assumptions, this yields an abstract characterization of controllability.
		For a parabolic equation with a bounded delay operator and a memory kernel
		admitting a finite exponential representation, we derive a global Carleman
		estimate by combining parabolic and ODE estimates with a temporal cutoff
		adapted to the delay. This leads to the required observability inequality and,
		through the Hilbert Uniqueness Method, to delay and memory-type null
		controllability. Numerical experiments illustrate the HUM construction and
		show that, for the tested configurations, moving control regions achieve
		lower control costs and improved terminal-state reduction compared with
		fixed control regions.
\end{abstract}


\end{frontmatter}

	
	\section{Introduction}
	\label{sec:intro}
	
	Controllability of evolution equations has been a central topic in control
	theory since the seminal work of Kalman on finite-dimensional linear systems
	\cite{kalman1960general}. In the infinite-dimensional setting, particularly for
	partial differential equations, the theory was developed through the pioneering
	contributions of Lions \cite{lions1988controlabilite}, Russell
	\cite{russell1978controllability}, Coron \cite{coron2007control}, and Zuazua
	\cite{zuazua2007controllability}. While controllability in finite dimensions is
	characterized by algebraic rank conditions, infinite-dimensional systems require
	fundamentally different analytical tools, including observability inequalities,
	Carleman estimates, and unique continuation principles.
	
	Evolution equations with memory effects have attracted considerable attention
	because of their relevance in modeling hereditary and relaxation phenomena.
	Classical diffusion models predict instantaneous propagation, which is
	incompatible with several physical processes. This motivated the introduction
	of memory-dependent constitutive laws, beginning with Gurtin's modification of
	Fourier's law \cite{gurtin1968general} and the Cattaneo equation
	\cite{cattaneo1958form}, and subsequently leading to a variety of models
	incorporating hereditary effects. The controllability of such systems has
	revealed significant differences from the classical parabolic theory.
	
	Systems with time delays have been studied independently and arise naturally
	in feedback-controlled processes, signal transmission, and actuation
	mechanisms. In many contemporary applications, however, memory and delay
	coexist and interact in a nontrivial manner. For instance, flexible aerospace
	structures may exhibit memory effects due to viscoelasticity, while digital
	feedback loops introduce time delays. Similar interactions occur in networked
	control systems and adaptive identification problems
	\cite{gallegos2024relaxed}. Consequently, analyzing memory or delay in
	isolation does not fully capture the underlying dynamical behavior, which
	motivates the development of a unified controllability framework.
	
	In this work, we investigate the controllability of evolution equations
	featuring both memory and delay effects. As a motivating example, we consider
	the parabolic system
	\begin{equation}\label{eq:main_system}
		\left\{
		\begin{aligned}
			y_t - \sum_{i,j=1}^{n} \left\{ \mathfrak{a}^{ij}(x)
			\left[ \mathfrak{a} y_{x_i}
			+ \int_0^t \mathfrak{b}(t-s,x)y_{x_i}(s,x)\,ds \right]
			\right\}_{x_j}
			&= u\chi_{\varpi}(x) \quad \text{in } Q, \\
			y &= 0 \quad \text{on } \Xi, \\
			y(0) &= y_0 \quad \text{in } \Theta,
		\end{aligned}
		\right.
	\end{equation}
	where $\Theta\subset\mathbb{R}^n$ is a bounded domain, $\varpi\subset\Theta$
	is the control region, $(\mathfrak{a}^{ij})$ is uniformly elliptic, and
	$\mathfrak{b}(\cdot,\cdot)$ is a smooth memory kernel.  Here, $Q=(0,T)\times\Theta$ and $\Xi=(0,T)\times\partial\Theta$.  
	
	When $\mathfrak{b}\equiv0$ and $\mathfrak{a}=1$, system~\eqref{eq:main_system}
	reduces to the classical heat equation. It is well known that this equation is
	null controllable in any positive time $T$ from any nonempty open subset
	$\varpi\subset\Theta$ \cite{fursikov1996controllability}. This property is
	closely related to the infinite speed of propagation of parabolic equations.
	In contrast, the presence of memory fundamentally changes the controllability
	picture. Even if the state satisfies $y(T)=0$, the system cannot in general be
	regarded as fully controlled unless the accumulated memory also vanishes:
	\begin{equation}\label{eq:memory_condition}
		\int_0^T \mathfrak{b}(T-s,x)y_{x_i}(s,x)\,ds \equiv 0.
	\end{equation}
	Otherwise, the memory term continues to influence the subsequent evolution of
	the system.
	
	Indeed, it has been shown that classical null controllability may fail for
	parabolic equations with nontrivial memory kernels when the control region is
	fixed \cite{guerrero2013remarks,halanay2012lack,pandolfi2013boundary}.
	This obstruction has motivated a broad range of investigations, including
	null controllability for wave equations with memory \cite{lu2017null},
	flow-based approaches \cite{wang2022flow}, observability results for heat
	equations with analytic memory kernels \cite{wang2024observability},
	time-optimal control \cite{zhou2023perturbations}, approximate controllability
	\cite{kumar2022approximate,zhou2014interior,barbu2000controllability}, and
	extensions to degenerate coupled systems \cite{allal2023null}. Despite this
	substantial progress, a complete controllability theory for systems involving
	memory effects remains largely open.

	The situation becomes more delicate when time-delay effects are incorporated
	alongside memory. Related difficulties arise in wave equations with
	time-varying delay \cite{liu2025memory}, semilinear parabolic equations with
	nonlocal delays \cite{casas2018measure}, and stochastic parabolic systems
	\cite{zhao2025null}. To capture this combined structure, we consider the
	abstract delayed--memory system
	\begin{equation}\label{eq:Intro_eq_7}
		\left\{
		\begin{aligned}
			y_t(t) &= \mathfrak{A}y(t) + \mathfrak{A}_1 y(t-h)
			+ \int_0^t \mathfrak{M}(t-s)y(s)\,ds + \mathfrak{B}(t)u(t),
			\quad t\in(0,T], \\
			y(\theta) &= \phi(\theta), \quad \theta\in[-h,0],
		\end{aligned}
		\right.
	\end{equation}
	formulated in a Hilbert space $Y$. Here, $y = y(t)$ is the state variable 
	taking values in $Y$, initial data satisfies
	$
	\phi \in \mathcal{H}_h := L^2(-h,0;Y)  \text{ with } 0 < h < T,
	$
	$\mathfrak{A}$ generates a $C_0$-semigroup, $\mathfrak{A}_1$ is bounded 
	(that is, $\mathfrak{A}_1\in\mathcal{L}(Y)$),
	$\mathfrak{M}\in L^2(0,T;\mathcal{L}(Y))$, and
	$\mathfrak{B}(\cdot)\in L^2(0,T;\mathcal{L}(U;Y))$. While null controllability
	for retarded parabolic systems has been studied
	\cite{khodja2014null,ma2025convergence}, the combined delay--memory case
	remains largely unexplored.

	The need for a more stringent notion of controllability becomes apparent when
	one examines the interaction between the delay and the state history in
	\eqref{eq:Intro_eq_7}. Even if a control steers the state to zero at the
	terminal time, $y(T)=0$, and simultaneously annihilates the memory term, the
	system may fail to remain at rest for $t>T$. To illustrate this phenomenon,
	consider the dynamics at $t=T+\varepsilon$, where $\varepsilon>0$ is
	arbitrarily small. Suppose that $y(T)=0$ and the memory integral vanishes,
	but that the state is not identically zero on the delay interval
	$[T-h,T)$; that is,
	\[
	y_t(T+\varepsilon)
	=
	\mathfrak{A}y(T+\varepsilon)
	+
	\mathfrak{A}_1 y(T+\varepsilon-h)
	+
	\int_0^{T+\varepsilon}
	\mathfrak{M}(T+\varepsilon-s)y(s)\,ds
	+
	\mathfrak{B}(T+\varepsilon)u(T+\varepsilon).
	\]
	While the terms involving $\mathfrak{A}$, the control, and the memory
	integral are initially negligible, the delayed term
	$\mathfrak{A}_1 y(T+\varepsilon-h)$ depends on the past state over
	$[T-h,T)$ and may act as an active source, thereby reintroducing energy into
	the system.
	
	This observation reveals a fundamental obstruction: guaranteeing that
	$y(t)=0$ for all $t>T$ requires more than nullifying the state and memory at
	the terminal time. One must additionally eliminate the delayed contribution
	generated by the terminal history, namely,
	\[
	\mathfrak{A}_1 y(t)=0
	\quad \text{for all } t\in[T-h,T),
	\]
	thereby preventing delayed reactivation and ensuring a genuinely stable rest
	state. This motivates the strengthened notion of controllability adopted in
	this paper.
	The appropriate phase space for the delayed system is
	$
	\mathcal H_h:= L^2(-h,0;Y).
	$
	We therefore introduce the following notion.
	
	\begin{definition}[Delay and Memory-Type Null Controllability]
		\label{def:Def_memory_type of null control}
		Let $\mathfrak A_1\in\mathcal L(Y)$ and $\widetilde{\mathfrak M}(\cdot)\in
		L^2(0,T;\mathcal L(Y))$ be a memory kernel, not necessarily the
		same as the kernel $\mathfrak M(\cdot)$ appearing in
		\eqref{eq:Intro_eq_7}. Let $\mathcal H_h$ denote the corresponding
		history space. The system~\eqref{eq:Intro_eq_7} is said to be
		\emph{delay and memory-type null controllable} with respect to the
		kernel $\widetilde{\mathfrak M}(\cdot)$ if, for every
		$\phi\in\mathcal H_h$, there exists a control
		$u\in L^2(0,T;U)$ such that the corresponding solution $y$
		satisfies
		\begin{equation}
			\label{eq:Intro_eq_9}
			\begin{cases}
				y(T)=0,\\[2pt]
				\mathfrak A_1y(t)=0
				\quad\text{for a.e. }t\in[T-h,T),\\[4pt]
				\displaystyle
				\int_0^T
				\widetilde{\mathfrak M}(T-s)y(s)\,ds=0.
			\end{cases}
		\end{equation}
	\end{definition}
	
	This definition generalizes both classical null controllability and
	memory-type null controllability \cite{chaves2017controllability}. The first
	condition in \eqref{eq:Intro_eq_9} eliminates the terminal state, the second
	eliminates the delayed contribution generated by the terminal history, and
	the third annihilates the accumulated memory component. Thus, the definition
	accounts simultaneously for the three mechanisms that may prevent the system
	from remaining at rest after the terminal time.
	
	By duality, delay and memory-type null controllability is connected with a
	strengthened observability inequality for the corresponding adjoint system.
	The derivation of such an inequality in the simultaneous presence of delay
	and memory constitutes the main analytical challenge of this work. In
	particular, the delayed term requires control of an appropriate terminal
	history of the adjoint variable, while the memory term leads to an additional
	nonlocal contribution. These effects are treated through an abstract
	duality framework followed by a moving-control Carleman analysis.
	
\medskip
\noindent\textbf{Main contributions.}
The main results of the paper can be summarized as follows.

\medskip
\noindent\textit{(i)}
We formulate the notion of \emph{delay and memory-type null controllability}
for evolution equations with simultaneous delay and memory effects. The
definition requires the terminal state, the delayed contribution generated by
the terminal history, and the prescribed terminal memory functional to vanish.
This identifies the additional terminal-history condition needed to prevent
delayed reactivation after the control time.

\medskip
\noindent\textit{(ii)}
We establish an abstract duality principle relating delay and memory-type null
controllability to an observability inequality for the corresponding
adjoint system on the extended interval $[0,T+h]$. The resulting observability
inequality involves both the initial adjoint state and the contribution of the
delayed history, while the terminal memory datum enters only as an auxiliary
dual variable. Under the stated unique continuation and well-posedness
assumptions, this provides an abstract characterization of
delay and memory-type null controllability.

\medskip
\noindent\textit{(iii)}
For a parabolic delayed--memory system on a bounded domain
$\Theta\subset\mathbb{R}^n$, with a bounded delay operator and a memory kernel
admitting a finite exponential representation, we derive a global Carleman
estimate adapted to a moving control region. The proof combines a parabolic
Carleman estimate for the delayed equation with ODE Carleman estimates for the
auxiliary memory variables and a temporal cutoff exploiting the positive
shift of the Carleman weight under the delay. The resulting estimate yields
the required observability inequality, and hence delay and memory-type null
controllability through the Hilbert Uniqueness Method.

\medskip
\noindent\textit{(iv)}
We complement the analytical results with numerical experiments for the
specialized case $\mathfrak A_1=I$ and an exponential memory kernel. The
computations implement the HUM construction and compare moving and fixed
control configurations  and
show that, for the tested configurations, moving control regions achieve
lower control costs and improved terminal-state reduction compared with
fixed control regions.

	\medskip
	\noindent\textbf{Organization.}
	Section~\ref{sec:Abstract Duality and Observability Analogue} presents the
	abstract framework and duality results. Section~\ref{sec:Observability}
	proves the main controllability results using moving controls.
	Section~\ref{sec:Application} discusses applications, and
	Section~\ref{sec:conclusion} concludes the paper.
	
	\section{Duality and the Abstract Observability Framework}
	\label{sec:Abstract Duality and Observability Analogue}
	
	This section establishes the dual characterization of delay and
	memory-type null controllability. The main point is that the terminal
	conditions
	\[
	\mathfrak A_1 y(t)=0,
	\qquad t\in[T-h,T),
	\text{ and }
	\int_0^T
	\widetilde{\mathfrak M}(T-s)y(s)\,ds=0
	\]
	are dual, respectively, to an observability term involving
	$\mathfrak A_1^*w$ on $[0,h]$ and to an auxiliary terminal datum
	associated with the memory constraint. The latter auxiliary datum is
	eliminated in the final observability estimate.
	We begin by considering the underlying system without memory and control inputs:
	\begin{equation}
		\label{eq:Dual_1}
		\left\{
		\begin{aligned}
			y'(t) &= \mathfrak{A}y(t) + \mathfrak{A_1} y(t - h), && 0 \leq t \leq T, \\
			y_0(\theta) &= \phi(\theta), && -h \leq \theta \leq 0.
		\end{aligned}
		\right.
	\end{equation}
	It is well-established that for any initial function $\phi  \in
	\mathcal H_h$, this system (\ref{eq:Dual_1}) possesses a unique solution $y(t)$. From this solution, we define the fundamental solution operator $S(t): Y \to Y$ as:
	\[
	S(t)\phi(0) =
	\begin{cases}
		y(t), & \text{if } t \geq 0, \\
		0, & \text{if } t < 0.
	\end{cases}
	\]
The operator family $\{S(t)\}_{t\geq0}$ is referred to as the
\emph{fundamental solution} of \eqref{eq:Dual_1};
see, for example, \cite{curtain2012introduction,wang2009approximate}. It satisfies the integral equation	
	\begin{align}\label{eq:Dual_2}
		S(t) &= T(t) + \int_0^t T(t - s) \mathfrak{A_1} S(s - h) \, ds, && t > 0, \notag \\
		S(0) &= I, \notag \\
		S(t) &= 0, && -h \leq t < 0,
	\end{align}
	where $\{T(t)\}_{t\geq0}$ denotes the $C_{0}$-semigroup generated by
	$\mathfrak{A}$.
	It follows from \cite{wang2005approximate} that \eqref{eq:Dual_2} admits
	a unique solution. Furthermore, there exists a constant
	$
	M_{1}=K_{0}\exp\!\left(K_{0}\|\mathfrak{A_1}\|T\right),
	$
	such that $\|S(t)\|\leq M_{1},
	 0\leq t\leq T,$
	where
	$
	\|T(t)\|\leq K_{0},
	 0\leq t\leq T.
	$
	
	The unique mild solution $y(\cdot) $ to the system \eqref{eq:Intro_eq_7} is given by the integral formulation found in \cite{wang2005approximate}:
	\begin{align}\label{eq:Dual_3}
		y(t) &= S(t)\phi(0) 
		+ \int_0^t S(t - s) \mathfrak{B}(s) u(s) \, ds \notag\\
		&\quad + \int_0^t S(t - s) \int_0^s \mathfrak{M}(s - \tau)y(\tau) \, d\tau \, ds, \quad 0 \leq t \leq T, \notag \\
		y_0(\theta) &= \phi(\theta), \quad -h \leq \theta < 0.
	\end{align}
	where
	$
	\phi
	\in
	\mathcal H_h
	:=
	 L^2(-h,0;Y).
	$
Thus, for admissible initial data
\(\phi\in\mathcal H_h\), the mild state satisfies
$
y\in L^2(-h,0;Y)\cap C([0,T];Y).
$		
	
	To develop the duality argument, we now introduce the adjoint system associated with \eqref{eq:Intro_eq_7}, which incorporates the memory kernel $\widetilde{\mathfrak{M}}(\cdot)$:
	\begin{definition}[Adjoint System]
		\label{eq:dual_eq_1}
		The adjoint system is solved backward from $t=T$. Its terminal data is $(w_T, \psi, z_T)$.
		\begin{equation}
			\label{eq:adjoint}
			\left\{
			\begin{aligned}
				-w_t(t) &= \mathfrak{A}^* w(t) + \mathfrak{A}_1^* w(t + h)\\
				&\quad + \int_t^T \mathfrak{M}(s - t)^* w(s)\,ds - \widetilde{\mathfrak{M}}(T - t)^* z_T, \quad t \in [0, T] \\
				w(T) &= w_T \\
				w(t) &= \psi(t), \quad t \in (T, T+h]
			\end{aligned}
			\right.
		\end{equation}

		(Note: $w(t+h) = \psi(t+h)$ is known data for $t \in [T-h, T]$ and is part of the solution $w$ for $t < T-h$.)
	\end{definition}
	We therefore introduce
	the extended terminal-data space
	\begin{equation}
		\label{eq:Xext}
		\mathcal X_{\rm ext}
		:=
		Y\times L^2(T,T+h;Y)\times Y,
	\end{equation}
	equipped with the norm
	\begin{equation}
		\label{eq:Xext-norm}
		\|(w_T,\psi,z_T)\|_{\mathcal X_{\rm ext}}^2
		:=
		\|w_T\|_Y^2
		+
		\|\psi\|_{L^2(T,T+h;Y)}^2
		+
		\|z_T\|_Y^2.
	\end{equation}
	Here $w_T$ is the terminal value of the adjoint equation,
	$\psi$ is the auxiliary terminal extension used to test the delayed
	history condition, and $z_T$ is the terminal datum associated with the
	memory constraint. The extension $\psi$ and the datum $z_T$ are
	auxiliary variables and are not included in the final observability
	quantity.

	\subsection{Abstract observability--controllability duality}
	
	The preceding identity shows that the natural dual quantity associated
	with the initial history is
	$
	\|w(0)\|_Y^2
	+
	\int_{-h}^{0}
	\|\mathfrak A_1^*w(t+h)\|_Y^2\,dt.
	$
	The terminal memory constraint does not produce an additional term in
	the final observability inequality because its dual datum $z_T$ is an
	auxiliary variable used to impose the corresponding terminal
	constraint.
	We can now state the abstract characterization.

	\begin{theorem}
		\label{prop:2.1}
		Assume that
		$
		\mathfrak A_1\in\mathcal L(Y),
		\mathfrak M\in L^2(0,T;\mathcal L(Y)),
		\widetilde{\mathfrak M}\in L^2(0,T;\mathcal L(Y)),
		$
		and
		$
		\mathfrak B(\cdot)
		\in L^2(0,T;\mathcal L(U,Y)),
		$
		and that the state equation \eqref{eq:Intro_eq_7} is well posed for
		initial data in $\mathcal H_h$.
		Then the system \eqref{eq:Intro_eq_7} is Delay and Memory-Type Null
		Controllable in the sense of
		Definition~\ref{def:Def_memory_type of null control}
		if and only if there exists a constant
		$\mathfrak K>0$ such that every admissible solution $w$ of
		the adjoint equation \eqref{eq:adjoint} satisfies
		\begin{equation}
			\label{eq:enhanced_obs}
			\|w(0)\|_Y^2
			+
			\int_0^h
			\|\mathfrak A_1^*w(t)\|_Y^2\,dt
			\le
			\mathfrak K
			\int_0^T
			\|\mathfrak B^*(t)w(t)\|_U^2\,dt .
		\end{equation}
		Equivalently,
		\begin{equation}
			\label{eq:enhanced_obs_shifted}
			\|w(0)\|_Y^2
			+
			\int_{-h}^{0}
			\|\mathfrak A_1^*w(t+h)\|_Y^2\,dt
			\le
			\mathfrak K
			\int_0^T
			\|\mathfrak B^*(t)w(t)\|_U^2\,dt .
		\end{equation}
	\end{theorem}
	
	\begin{proof}
		We divide the proof into sufficiency and necessity.
	First, we establish the duality identity by integrating
	$\frac{d}{dt}\langle w(t),y(t)\rangle_Y$ from $0$ to $T$. Using
	the state equation~\eqref{eq:Intro_eq_7} and the adjoint
	system~\eqref{eq:adjoint}, we obtain
	\begin{align}
		\label{eq:full_dual_id}
		\langle w(T),y(T)\rangle_Y-\langle w(0),y(0)\rangle_Y &= \int_0^T \frac{d}{dt}\langle w(t),y(t)\rangle_Y\,dt\nonumber\\
		& = \int_0^T \langle w_t(t),y(t)\rangle_Y\,dt+ \int_0^T \langle w(t),y_t(t)\rangle_Y\,dt \nonumber\\
		&=
		\int_0^T
		\langle \mathfrak B^*(t)w(t),u(t)\rangle_U\,dt
		+\int_{-h}^{0}
		\langle
		\mathfrak A_1^*w(t+h),\phi(t)
		\rangle_Y\,dt
		\nonumber\\
		&-\int_{T-h}^{T}
		\langle
		\mathfrak A_1^*w(t+h),y(t)
		\rangle_Y\,dt
		+\int_0^T
		\langle
		z_T,\widetilde{\mathfrak M}(T-t)y(t)
		\rangle_Y\,dt .
	\end{align}
	Here, we have used $y(t)=\phi(t)$ for $t\in[-h,0]$,
	$w(T)=w_T$, and the terminal extension
	$w(t+h)=\psi(t+h)$ for $t\in[T-h,T)$.
	Hence,
	\begin{align}
		\label{eq:full_dual_identity_final}
		\langle w_T,y(T)\rangle_Y
		-\langle w(0),\phi(0)\rangle_Y
		={}&
		\int_0^T
		\langle \mathfrak B^*(t)w(t),u(t)\rangle_U\,dt
		+\int_{-h}^{0}
		\langle
		\mathfrak A_1^*w(t+h),\phi(t)
		\rangle_Y\,dt
		\nonumber\\
		&-\int_{T-h}^{T}
		\langle
		\psi(t+h),\mathfrak A_1y(t)
		\rangle_Y\,dt
		+\int_0^T
		\langle
		z_T,\widetilde{\mathfrak M}(T-t)y(t)
		\rangle_Y\,dt .
	\end{align}
	
	If $y$ satisfies the three terminal conditions in
	Definition~\ref{def:Def_memory_type of null control}, namely
	\[
	y(T)=0,\quad
	\mathfrak A_1y(t)=0
	\quad\text{for a.e. }t\in[T-h,T),\text{ and }
	\int_0^T
	\widetilde{\mathfrak M}(T-t)y(t)\,dt=0,
	\]
	then all the corresponding terminal terms in
	\eqref{eq:full_dual_identity_final} vanish. Consequently,
	\begin{equation}
		\label{eq:fundamental-duality}
		\int_0^T
		\langle
		\mathfrak B^*(t)w(t),u(t)
		\rangle_U\,dt
		=
		-\langle w(0),\phi(0)\rangle_Y
		-\int_{-h}^{0}
		\langle
		\mathfrak A_1^*w(t+h),\phi(t)
		\rangle_Y\,dt ,
	\end{equation}	
		
	this is known as duality identity.

		\medskip
		\noindent
		\textbf{Step 1: Sufficiency.}
		
		Assume that the observability estimate
		\eqref{eq:enhanced_obs} holds. Fix
		$
		\phi\in\mathcal H_h.
		$
		Let
		$
		\mathcal L
		:=
		\left\{
		\mathfrak B^*(\cdot)w(\cdot):
		w\text{ solves }\eqref{eq:adjoint}
		\right\}
		\subset L^2(0,T;U).
		$
		Define
		$
		\Lambda:\mathcal L\to\mathbb R
		$
		by
		\begin{equation}
			\label{eq:Lambda}
			\Lambda(\mathfrak B^*w)
			:=
			-\langle w(0),\phi(0)\rangle_Y
			-
			\int_{-h}^{0}
			\langle
			\mathfrak A_1^*w(t+h),\phi(t)
			\rangle_Y\,dt .
		\end{equation}
		
		We first verify that $\Lambda$ is well defined. Suppose that
		\[
		\mathfrak B^*(t)w(t)=0
		\qquad\text{for a.e. }t\in(0,T).
		\]
		Then \eqref{eq:enhanced_obs} implies
		$
		w(0)=0
		$
		and
		$
		\mathfrak A_1^*w(t+h)=0
		$
		for a.e. $t\in[-h,0]$. Hence
		$
		\Lambda(\mathfrak B^*w)=0.
		$
		Moreover, by the Cauchy--Schwarz inequality,
		\begin{align}
			|\Lambda(\mathfrak B^*w)|^2
			\le{}&
			C_\phi
			\left[
			\|w(0)\|_Y^2
			+
			\int_{-h}^{0}
			\|\mathfrak A_1^*w(t+h)\|_Y^2\,dt
			\right]
			\nonumber\\
			\le{}&
			C_\phi\mathfrak K
			\int_0^T
			\|\mathfrak B^*(t)w(t)\|_U^2\,dt .
			\label{eq:Lambda-bound}
		\end{align}
		Thus $\Lambda$ extends continuously to
		$\overline{\mathcal L}$.
		By the Riesz representation theorem, there exists
		$
		u_\phi\in L^2(0,T;U)
		$
		such that
		\begin{equation}
			\label{eq:Riesz}
			\int_0^T
			\langle
			\mathfrak B^*(t)w(t),u_\phi(t)
			\rangle_U\,dt
			=
			\Lambda(\mathfrak B^*w)
		\end{equation}
		for every admissible adjoint solution $w$.
		Let $y_\phi$ be the corresponding state solution. Combining
		\eqref{eq:Riesz} with \eqref{eq:full_dual_identity_final} gives
		\begin{align}
			\label{eq:terminal_identity}
			0={}&
			\langle w_T,y_\phi(T)\rangle_Y
			-
			\int_{T-h}^{T}
			\langle
			\psi(t+h),\mathfrak A_1 y_\phi(t)
			\rangle_Y\,dt
			+
			\int_0^T
			\left\langle
			z_T,
			\widetilde{\mathfrak M}(T-t)y_\phi(t)
			\right\rangle_Ydt .
		\end{align}
			Choosing independently $w_T$, $\psi$, and $z_T$ yields
		\[
		y_\phi(T)=0,\qquad
		\mathfrak A_1y_\phi(t)=0\ \text{a.e. on }[T-h,T),
		\qquad
		\int_0^T\widetilde{\mathfrak M}(T-t)y_\phi(t)\,dt=0.
		\]
		Thus all three conditions in
		Definition~\ref{def:Def_memory_type of null control} are satisfied.
		Therefore the system is Delay and Memory-Type Null Controllable.
		
		\medskip
		\noindent
		\textbf{Step 2: Necessity.}
		
		Assume now that the system is Delay and Memory-Type Null Controllable.
		Let
		$
		\mathcal U=L^2(0,T;U).
		$
		For $(\phi,u)\in\mathcal H_h\times\mathcal U$, let
		$y_{\phi,u}$ denote solution of the corresponding state.
		Define the terminal constraint operator
		\[
		\mathcal C:
		\mathcal H_h\times\mathcal U
		\longrightarrow
		\mathcal Z
		\]
		by
		\begin{equation}
			\label{eq:terminal-constraint}
			\mathcal C(\phi,u)
			=
			\left(
			y_{\phi,u}(T),
			\mathfrak A_1y_{\phi,u}|_{[T-h,T)},
			\int_0^T
			\widetilde{\mathfrak M}(T-s)y_{\phi,u}(s)\,ds
			\right),
		\end{equation}
		where
		\[
		\mathcal Z
		=
		Y\times L^2(T-h,T;Y)\times Y.
		\]
		By well-posedness, $\mathcal C$ is bounded and linear.
		Delay and Memory-Type Null Controllability implies that for every
		$\phi\in\mathcal H_h$ there exists $u\in\mathcal U$ such that
		$
		\mathcal C(\phi,u)=0.
		$
		Hence the projection
		\[
		\mathcal P:\ker\mathcal C\longrightarrow\mathcal H_h,
		\qquad
		\mathcal P(\phi,u)=\phi,
		\]
		is surjective. Since $\ker\mathcal C$ is a closed subspace of
		$\mathcal H_h\times\mathcal U$, the open mapping theorem yields
		a constant $C_0>0$ such that, for every
		$\phi\in\mathcal H_h$, one can choose $u_\phi\in\mathcal U$
		satisfying
		\begin{equation}
			\label{eq:bounded-control}
			\mathcal C(\phi,u_\phi)=0,
			\qquad
			\|u_\phi\|_{\mathcal U}
			\le
			C_0\|\phi\|_{\mathcal H_h}.
		\end{equation}
		
		Suppose, by contradiction, that \eqref{eq:enhanced_obs} fails.
		Then, for every $k\in\mathbb N$, there exists an admissible adjoint
		solution $w^k$ such that
		\begin{equation}
			\label{eq:failure-observability}
			0\le
			\int_0^T
			\|\mathfrak B^*(t)w^k(t)\|_U^2\,dt
			<
			\frac1{k^2}
			\left[
			\|w^k(0)\|_Y^2
			+
			\int_{-h}^{0}
			\|\mathfrak A_1^*w^k(t+h)\|_Y^2\,dt
			\right].
		\end{equation}
		
		Set
		$
		N_k^2
		:=
		\|w^k(0)\|_Y^2
		+
		\int_{-h}^{0}
		\|\mathfrak A_1^*w^k(t+h)\|_Y^2\,dt>0
		$
		and define
		$
		\widetilde w^k
		:=
		\frac{\sqrt{k}}{N_k}w^k.
		$
		By linearity,
		\begin{equation}
			\label{eq:normalized}
			\|\widetilde w^k(0)\|_Y^2
			+
			\int_{-h}^{0}
			\|\mathfrak A_1^*\widetilde w^k(t+h)\|_Y^2\,dt
			=
			k,
		\end{equation}
		whereas \eqref{eq:failure-observability} yields
		\begin{equation}
			\label{eq:small-observation}
			\int_0^T
			\|\mathfrak B^*(t)\widetilde w^k(t)\|_U^2\,dt
			<
			\frac1k.
		\end{equation}
		
		For arbitrary
		$
		\phi\in\mathcal H_h,
		$
		choose $u_\phi$ according to \eqref{eq:bounded-control}. Since
		$
		\mathcal C(\phi,u_\phi)=0,
		$
		the duality identity \eqref{eq:fundamental-duality} gives
		\begin{align}
			\mathcal G_k(\phi)
			:={}&
			\langle
			\widetilde w^k(0),\phi(0)
			\rangle_Y
			+
			\int_{-h}^{0}
			\langle
			\mathfrak A_1^*\widetilde w^k(t+h),
			\phi(t)
			\rangle_Y\,dt
			\nonumber\\
			={}&
			-\int_0^T
			\langle
			\mathfrak B^*(t)\widetilde w^k(t),
			u_\phi(t)
			\rangle_U\,dt .
		\end{align}
		
		Using Cauchy--Schwarz, \eqref{eq:small-observation}, and
		\eqref{eq:bounded-control}, we obtain
		\begin{align}
			|\mathcal G_k(\phi)|
			&\le
			\left(
			\int_0^T
			\|\mathfrak B^*(t)\widetilde w^k(t)\|_U^2\,dt
			\right)^{1/2}
			\|u_\phi\|_{\mathcal U}
			\nonumber\\
			&\le
			\frac{C_0}{\sqrt{k}}
			\|\phi\|_{\mathcal H_h}.
		\end{align}
		Therefore
		\begin{equation}
			\label{eq:Gk-bound}
			\|\mathcal G_k\|_{\mathcal H_h^*}
			\le
			\frac{C_0}{\sqrt{k}}.
		\end{equation}
		
		On the other hand,
		$
		\mathcal H_h^*
		=
		 L^2(-h,0;Y),
		$
		and $\mathcal G_k$ is represented by
		$
		\left(
		\widetilde w^k(0),
		\mathfrak A_1^*\widetilde w^k(\,\cdot+h)
		\right).
		$
		Consequently,
		\begin{equation}
			\label{eq:Gk-norm}
			\|\mathcal G_k\|_{\mathcal H_h^*}^2
			=
			\|\widetilde w^k(0)\|_Y^2
			+
			\int_{-h}^{0}
			\|\mathfrak A_1^*\widetilde w^k(t+h)\|_Y^2\,dt
			=k
		\end{equation}
		by \eqref{eq:normalized}. Combining
		\eqref{eq:Gk-bound} and \eqref{eq:Gk-norm}, we obtain
		\[
		k
		=
		\|\mathcal G_k\|_{\mathcal H_h^*}^2
		\le
		\frac{C_0^2}{k},
		\]
		which is impossible for sufficiently large $k$.
		Hence the observability estimate \eqref{eq:enhanced_obs} must hold.
		This proves the necessity and completes the proof.
	\end{proof}
	
	\begin{remark}
		\label{rem:2.3}
		The observability estimate \eqref{eq:enhanced_obs} contains exactly
		the two quantities that are dual to the initial state and the delayed
		history:
		$
		\|w(0)\|_Y^2
		\text{ and }
		\int_0^h
		\|\mathfrak A_1^*w(t)\|_Y^2\,dt.
		$
		No separate norm of the auxiliary terminal extension or of the
		memory dual datum $z_T$ is required in the final observability
		inequality.
		
		The terminal extension of the adjoint trajectory is used only as an
		auxiliary test in the duality argument for the condition
		$\mathfrak A_1y=0$ on $[T-h,T)$. Similarly, $z_T$ is an auxiliary
		dual variable used to impose the terminal memory condition
		$
		\int_0^T
		\widetilde{\mathfrak M}(T-s)y(s)\,ds=0.
		$
		Neither auxiliary variable appears in the final observability
		inequality.
		
		This formulation is particularly useful in the parabolic setting,
		where $\mathfrak B(t)$ is generated by a moving spatial control
		region. The observability estimate \eqref{eq:enhanced_obs} is the
		abstract estimate that will be established in Section~\ref{sec:Observability} by means of
		the Carleman inequalities.
	\end{remark}

\section{Null Controllability of Parabolic Equations with Delay and
	Memory Using Moving Controls}
\label{sec:Observability}

In this section, we investigate Delay and Memory--type null
controllability for a class of parabolic equations with a bounded
delay and a Volterra memory term. We begin with the classical example
\begin{equation}
	\label{eq:59}
	\left\{
	\begin{aligned}
		y_t-\Delta y+a\int_0^t y(s)\,ds
		&=u\chi_{\varpi}(x)
		&&\text{in }(0,T)\times\Theta,\\
		y&=0
		&&\text{on }(0,T)\times\partial\Theta,\\
		y(0)&=y_0
		&&\text{in }\Theta.
	\end{aligned}
	\right.
\end{equation}
with $a\in\mathbb R$.
If $\varpi=\Theta$, the memory term can be compensated by the
control. However, when $\varpi\subsetneq\Theta$, it is known
\cite{chaves2017controllability,guerrero2013remarks,
	halanay2012lack,pandolfi2013boundary}
that \eqref{eq:59} is null controllable if and only if $a=0$.
Thus, fixed controls fail in the presence of memory, motivating the
classical idea of moving controls
\cite{chaves2014null,russell1978controllability,
	martin2013null,khapalov2001mobile}.

The system considered below is controlled on a time-dependent region
$\varpi(\cdot)\subset\Theta$; see \eqref{eq:Zxt}.
Let $\mathfrak A_1\in\mathcal L(L^2(\Theta))$. We write
\begin{equation}
	\label{eq:A1-bounded-section3}
	K:=
	\|\mathfrak A_1\|_{\mathcal L(L^2(\Theta))}
	=
	\|\mathfrak A_1^*\|_{\mathcal L(L^2(\Theta))}.
\end{equation}
Consequently,
\begin{equation}
	\label{eq:A1-adjoint-bound-section3}
	\|\mathfrak A_1^*v\|_{L^2(\Theta)}
	\leq K\|v\|_{L^2(\Theta)},
	\qquad
	v\in L^2(\Theta).
\end{equation}

The parabolic system is 
\begin{equation}
	\label{eq:60}
	\left\{
	\begin{aligned}
		y_t-\Delta y(t)-\mathfrak A_1y(t-h)
		-\int_0^t\mathfrak M(t-s)y(s)\,ds
		&=u\chi_{\varpi(t)}(x)
		&&\text{in }Q,\\
		y&=0
		&&\text{on }\Xi,\\
		y(\theta)&=\phi(\theta),
		&&-h\leq\theta\leq0,\quad\text{in }\Theta.
	\end{aligned}
	\right.
\end{equation}

Here $\mathfrak A_1$ is bounded on $L^2(\Theta)$. Therefore, the
delayed state is treated as an $L^2(\Theta)$-valued lower-order term,
and no spatial derivative of the delayed state is required.

The moving control region $\varpi(t)$ is generated by the flow
$Z(x,t,t_0)$ satisfying
\begin{equation}
	\label{eq:Zxt}
	\left\{
	\begin{aligned}
		\frac{\partial Z(x,t,t_0)}{\partial t}
		&=f(t,Z(x,t,t_0)),
		&&t\in[0,T],\\
		Z(x,t_0,t_0)&=x,
		&&x\in\mathbb R^n.
	\end{aligned}
	\right.
\end{equation}
where
$f\in C([0,T];W^{2,\infty}(\mathbb R^n;\mathbb R^n))$.
We set $\varpi(t)=Z(\varpi_0,t,0).$

\subsection{Geometric Assumption and Carleman Weights}

Following the moving-control construction of
Chaves--Silva--Rosier--Zuazua
\cite{chaves2014null,chaves2017controllability}, we impose the following geometric
hypothesis.

\begin{assumption}
	\label{Ass:Assum_Obser_1}
	We assume the existence of a flow $Z(x,t,0)$ generated by a vector
	field
	$f\in C([0,T];W^{2,\infty}(\mathbb R^n;\mathbb R^n))$. Additionally, we presuppose a bounded, open, and smooth set
	$\varpi_0\subset\mathbb R^n$, a smooth curve
	$\Gamma(\cdot)\in C^\infty([0,T];\mathbb R^n)$, and two distinct
	times $t_1$ and $t_2$ with $0\leq t_1<t_2\leq T$ that collectively
	satisfy
	\begin{equation}
		\left\{
		\begin{aligned}
			&\Gamma(t)\in
			Z(\varpi_0,t,0)\cap\Theta
			\qquad \forall\,t\in[0,T],\\
			&\overline{\Theta}\subset
			\bigcup_{t\in[0,T]}Z(\varpi_0,t,0),\\
			&\Theta\setminus
			\overline{Z(\varpi_0,t,0)}
			\text{ is nonempty and connected for }
			 t \in[0,t_1]\cup[t_2,T],\\
			&\Theta\setminus
			\overline{Z(\varpi_0,t,0)}
			\text{ has two disjoint, nonempty, connected components for }
			 t\in(t_1,t_2),\\
			&\forall\gamma(\cdot)\in C([0,T];\Theta),\quad
			\exists\bar t\in[0,T]\text{ such that }
			\gamma(\bar t)\in Z(\varpi_0,\bar t,0).
		\end{aligned}
		\right.
	\end{equation}
\end{assumption}

The geometric condition is precisely the type used in the construction
of the moving Carleman weight in
\cite[Lemma~4.3]{chaves2014null}.

\begin{lemma}[Moving-control weight construction]
	\label{lemma:lemma_Observation_1}
	If Assumption~\ref{Ass:Assum_Obser_1} is satisfied, and if
	$\omega_0,\omega_1$ are two nonempty open subsets of $\mathbb R^n$
	such that
	$\overline{\omega_0}\subset\omega_1$ and
	$\overline{\omega_1}\subset\varpi$, then there exist
	$\delta\in(0,T/2)$ and $\rho\in C^\infty(\overline Q)$ such that
	\begin{equation}
		\left\{
		\begin{array}{ll}
			\nabla\rho(t,x)\neq0,
			&t\in[0,T],\
			x\in\overline{\Theta}\setminus Z(\omega_1,t,0),\\[1mm]
			\rho_t(t,x)\neq0,
			&t\in[0,T],\
			x\in\overline{\Theta}\setminus Z(\omega_1,t,0),\\[1mm]
			\rho_t(t,x)>0,
			&t\in[0,\delta],\
			x\in\overline{\Theta}\setminus Z(\omega_1,t,0),\\[1mm]
			\rho_t(t,x)<0,
			&t\in[T-\delta,T],\
			x\in\overline{\Theta}\setminus Z(\omega_1,t,0),\\[1mm]
			\partial_\nu\rho(t,x)\leq0,
			&t\in[0,T],\ x\in\partial\Theta,\\[1mm]
			\rho(t,x)>
			\dfrac34\|\rho\|_{L^\infty(Q)},
			&t\in[0,T],\ x\in\overline{\Theta}.
		\end{array}
		\right.
	\end{equation}
\end{lemma}

Throughout this section, $C>0$ denotes a generic positive constant
which may change from line to line and is independent of the functions
under consideration.

\subsection{Delay-compatible Carleman weight}

\begin{assumption}[Delay-compatible Carleman weight with temporal cutoff]
	\label{assumption:carleman-weight}
	Let $\Theta\subset\mathbb R^n$ be a bounded $C^\infty$ domain, and
	fix $T>0$ and $0<h<T$. 
	Let $\rho\in C^\infty(\overline Q)$ be the function provided by
	Lemma~\ref{lemma:lemma_Observation_1}. Choose $\lambda>0$
	sufficiently large and let
	$C_0>e^{\frac32\lambda\|\rho\|_{L^\infty(Q)}}$.
	Define $\vartheta(t,x):=C_0-e^{\lambda\rho(t,x)}.$
	After increasing $C_0$ if necessary, assume that
	$\vartheta(t,x)\ge c_0>0$ on $\overline Q$.
	Let
	\begin{equation}
		g(t)=\exp\left(\frac{\kappa}{T-t}\right),
		\qquad 0\leq t<T,
	\end{equation}
	for some $\kappa>0$. Thus
	$g\in C^\infty([0,T))$, $g'(t)>0$, and
	$g(t)\to+\infty$ as $t\to T^-$.
	Define
	\begin{equation}
		\label{eq:3.4}
		\phi(t,x)=g(t)\vartheta(t,x),
		\qquad
		\xi(t,x)=g(t)e^{\lambda\rho(t,x)}.
	\end{equation}
	
	Assume that the parameters in the construction are chosen so that
	there exists $c_d>0$ such that
	\begin{equation}
		\label{eq:multiplicative-lag-gap}
		\phi(t+h,x)-\phi(t,x)
		\ge c_dg(t)>0,
		\qquad
		0\leq t<T-h,\quad x\in\overline\Theta.
	\end{equation}
	
	Fix $\delta\in(0,(T-h)/2)$ and choose
	$\zeta\in C^\infty([0,T])$ satisfying
	\begin{equation}
		\label{eq:temporal-cutoff}
		0\leq\zeta\leq1,\qquad
		\zeta(t)=0\quad\text{for }0\leq t\leq h+\delta,
		\qquad
		\zeta(t)=1\quad\text{for }h+2\delta\leq t<T.
	\end{equation}
	Consequently,
	\begin{equation}
		\label{eq:delay-cutoff-support}
		\operatorname{supp}\zeta'\subset[h+\delta,h+2\delta],
		\qquad
		\zeta(t+h)\neq0\Longrightarrow t\geq\delta.
	\end{equation}
\end{assumption}

\subsection{Carleman functionals}

For $p\in H^{1,2}(Q)$ and $s>0$, introduce
\begin{equation}
	\label{eq:IH-definition}
	\begin{aligned}
		I_H(p):={}&
		\iint_Q
		\Big[
		(s\xi)^{-1}|p_t|^2
		+(s\xi)^{-1}|\Delta p|^2
		+s\lambda^2\xi|\nabla p|^2
		+s^3\lambda^4\xi^3|p|^2
		\Big]e^{-2s\phi}\,dx\,dt .
	\end{aligned}
\end{equation}

For a function $q\in L^2(Q)$ satisfying a first-order equation in time, set
\begin{equation}
	\label{eq:IO-definition}
	I_O(q):=
	\iint_Q
	\Big[
	(s\xi)^{-1}|q_t|^2
	+(\lambda s\xi)^2|q|^2
	\Big]e^{-2s\phi}\,dx\,dt .
\end{equation}

We also define the short-time remainder
\begin{equation}
	\label{eq:Rsh-definition}
	\mathcal R_{\rm sh}(p):=
	\int_0^{h+2\delta}
	\|p(t)\|_{L^2(\Theta)}^2\,dt .
\end{equation}

Since $\mathfrak A_1$ is bounded on $L^2(\Theta)$, this remainder
requires only $L^2$-regularity of the prescribed history.

\begin{lemma}[Carleman estimate for the bounded-delay equation]
	\label{lemma:carleman-parabolic-delay}
	Suppose that Lemma~\ref{lemma:lemma_Observation_1} and
	Assumption~\ref{assumption:carleman-weight} hold. Let
	$\mathfrak A_1\in\mathcal L(L^2(\Theta))$ and
	$f\in L^2(0,T;L^2(\Theta))$. If $p$ is a weak solution of
	\begin{equation}
		\label{eq:delay-parabolic-3}
		p_t-\Delta p-\mathfrak A_1^*p(t-h)=f
		\qquad\text{in }Q,
	\end{equation}
	with $p=0$ on $\Xi$ and with prescribed history
	$p|_{[-h,0]}=\varphi\in L^2(-h,0;L^2(\Theta))$,
	then, with $v=\zeta p$ and $\zeta$ as in
	\eqref{eq:temporal-cutoff}, there exist $\lambda_0>0$ and
	$s_0(\lambda)>0$ such that, for all $\lambda\geq\lambda_0$ and
	$s\geq s_0(\lambda)$,
	\begin{equation}
		\label{eq:delay-carl-final}
		\begin{aligned}
			I_H(v)\leq{}&
			C\iint_Q|f|^2e^{-2s\phi}\,dx\,dt
			+
			C\iint_{(0,T)\times\varpi(t)}
			\lambda(\lambda s\xi)^3|p|^2e^{-2s\phi}\,dx\,dt\\
			&+
			C\mathcal R_{\rm sh}(p).
		\end{aligned}
	\end{equation}
	Moreover,
	\begin{equation}
		\label{eq:p-weighted-by-cutoff}
		\iint_Q|p|^2e^{-2s\phi}\,dx\,dt
		\leq
		C\mathcal R_{\rm sh}(p)
		+\frac{C}{s^3\lambda^4}I_H(v).
	\end{equation}
\end{lemma}

\begin{proof}
	Choose $\zeta$ as in \eqref{eq:temporal-cutoff} and set $v=\zeta p$.
	Since $\zeta$ depends only on time,
	$$v_t-\Delta v
		=
		\zeta f
		+
		\zeta\mathfrak A_1^*p(t-h)
		+
		\zeta'p.$$
	In particular, no spatial derivative of the delayed term is
	needed because $\mathfrak A_1$ is bounded on $L^2(\Theta)$.	
	The proof combines the classical parabolic Carleman estimate
	of Fursikov and Imanuvilov
	\cite[Chapter~I, \S1]{fursikov1996controllability}
	with the moving-control weight construction
	\cite[Lemma~4.4]{chaves2014null}.
 to $v$ gives
	\begin{align}
		I_H(v)
		\leq{}&
		C\iint_Q|\zeta f|^2e^{-2s\phi}\,dx\,dt
		+
		C\iint_Q
		|\zeta\mathfrak A_1^*p(t-h)|^2e^{-2s\phi}\,dx\,dt
		\nonumber\\
		&+
		C\iint_Q|\zeta'p|^2e^{-2s\phi}\,dx\,dt
		+
		C\iint_{(0,T)\times\varpi(t)}
		\lambda(\lambda s\xi)^3|p|^2e^{-2s\phi}\,dx\,dt.
		\label{eq:delay-carl-before-absorption}
	\end{align}
	
	For the delayed term, using
	\eqref{eq:A1-adjoint-bound-section3} and the change of variables
	$\tau=t-h$, we obtain
	\begin{align*}
		&\iint_Q
		|\zeta(t)\mathfrak A_1^*p(t-h)|^2e^{-2s\phi(t,x)}
		\,dx\,dt
		\nonumber\\
		&\qquad\leq
		K^2\int_0^{T-h}\int_\Theta
		|\zeta(\tau+h)|^2|p(\tau,x)|^2
		e^{-2s\phi(\tau+h,x)}\,dx\,d\tau .
	\end{align*}
	By \eqref{eq:delay-cutoff-support}, $\zeta(\tau+h)\neq 0$ implies
	$\tau\geq\delta$. Put
	$
	a_d:=c_dg(\delta)>0.
	$
	Then \eqref{eq:multiplicative-lag-gap} yields
	$
	e^{-2s\phi(\tau+h,x)}
	\leq
	e^{-2sa_d}e^{-2s\phi(\tau,x)}.
	$	
We split the resulting integral into
$
[\delta,\min\{h+2\delta,T-h\}]
\text{ and }
[h+2\delta,T-h],
$
where the second interval is understood to be empty if
$h+2\delta\ge T-h$.
The first contribution is bounded by $CR_{\rm sh}(p)$.
On the second interval, $\zeta(\tau)=1$,  and therefore $p(\tau)=v(\tau)$. Hence

	\begin{align*}
		&\iint_Q
		|\zeta\mathfrak A_1^*p(t-h)|^2e^{-2s\phi}\,dx\,dt
		\leq
		C\mathcal R_{\rm sh}(p)
		+
		K^2e^{-2sa_d}
		\iint_Q|v|^2e^{-2s\phi}\,dx\,dt .
	\end{align*}
	
	Since $g(t)\ge g(0)>0$ and $\rho$ is bounded on $Q$, there exists
	$c_\xi>0$, independent of $s$, such that
	$
	\xi(t,x)\ge c_\xi>0,\text{ for all } (t,x)\in Q.
	$
	Hence, from the zero-order term in $I_H(v)$,
	$
	\iint_Q |v|^2e^{-2s\varphi}\,dxdt
	\le
	\frac{C}{s^3\lambda^4}I_H(v).
	$
	Consequently, after fixing $\lambda\geq\lambda_0$ and then taking
	$s\geq s_0(\lambda)$ sufficiently large,
	$
	K^2e^{-2sa_d}\frac{C}{s^3\lambda^4}
	\leq\frac14.
	$
	Thus,
	\begin{equation}
		\label{eq:delay-absorption}
		\iint_Q
		|\zeta\mathfrak A_1^*p(t-h)|^2e^{-2s\phi}
		\leq
		\frac14 I_H(v)+C\mathcal R_{\rm sh}(p).
	\end{equation}
	
	Moreover,
	\begin{equation}
		\label{eq:cutoff-remainder}
		\iint_Q|\zeta'p|^2e^{-2s\phi}\,dx\,dt
		\leq C\mathcal R_{\rm sh}(p),
	\end{equation}
	because $\operatorname{supp}\zeta'
	\subset[h+\delta,h+2\delta]$.
	Combining \eqref{eq:delay-carl-before-absorption},
	\eqref{eq:delay-absorption}, and
	\eqref{eq:cutoff-remainder}, and absorbing the
	$I_H(v)$ contribution proves \eqref{eq:delay-carl-final}.
	Finally, since $0\leq\zeta\leq1$ and $\zeta=1$ on
	$[h+2\delta,T)$,
	\[
	\iint_Q|p|^2e^{-2s\phi}
	\leq C
	\mathcal R_{\rm sh}(p)
	+
	\iint_Q|v|^2e^{-2s\phi}.
	\]
	The zero-order term in $I_H(v)$ then gives
	\eqref{eq:p-weighted-by-cutoff}.
\end{proof}

\begin{lemma}[ODE Carleman estimate]
	\label{lemma:carleman-ODE}
	Suppose that Lemma~\ref{lemma:lemma_Observation_1} and
	Assumption~\ref{assumption:carleman-weight} hold. Let $\mu>0$ and
	$q$ satisfy
	\begin{equation}
		\label{eq:ODE-equation}
		q_t+\mu q=G
		\qquad\text{in }Q,
		\qquad
		q(0)=0.
	\end{equation}
	Then there exist $\lambda_0>0$ and $s_0(\lambda)>0$ such that,
	for $\lambda\geq\lambda_0$ and $s\geq s_0(\lambda)$,
	\begin{equation}
		\label{eq:ODE-carl-final}
		I_O(q)
		\leq
		C\iint_Q|G|^2e^{-2s\phi}\,dx\,dt
		+
		C\iint_{(0,T)\times\varpi(t)}
		(\lambda s\xi)^2|q|^2e^{-2s\phi}\,dx\,dt .
	\end{equation}
\end{lemma}

\begin{proof}	
	The argument is based on the classical parabolic Carleman
	inequality of Fursikov and Imanuvilov
	\cite[Chapter~I, \S1]{fursikov1996controllability}.
	The estimate follows from the moving-control ODE Carleman estimate
	\cite[Lemma~4.5]{chaves2014null}. Applying that estimate with the
	weight constructed in Lemma~\ref{lemma:lemma_Observation_1}, the
	initial contribution vanishes because $q(0)=0$. The sign properties
	of the temporal part of the weight near $t=0$ and $t=T$ provide the
	required coercivity. This yields \eqref{eq:ODE-carl-final}.
\end{proof}

Assume that the memory kernel appearing in the state equation has the
finite exponential representation
\begin{equation}
	\label{eq:memory-kernel-exp}
	\mathfrak M(t)
	=
	\sum_{j=1}^{J}\alpha_j e^{-\mu_jt},
	\qquad
	\mu_j>0,
\end{equation}
where $J\in\mathbb N$ and $\alpha_j\in\mathbb R$.
We emphasize that no relation
$
\mathfrak M=\widetilde{\mathfrak M}
$
is imposed. In particular, $\widetilde{\mathfrak M}$ is not required
to admit the exponential representation \eqref{eq:memory-kernel-exp}.
The kernel $\widetilde{\mathfrak M}$ enters only through the terminal
memory functional and the corresponding terminal forcing in the
adjoint system.

\begin{lemma}[Finite-dimensional realization]
	\label{lemma:memory-realization}
	For $p\in L^2(Q)$ define
	\begin{equation}
		\label{eq:qj-definition}
		q_j(t,x)
		=
		\int_0^t e^{-\mu_j(t-r)}p(r,x)\,dr,
		\qquad j=1,\ldots,J.
	\end{equation}
	Then
	\begin{equation}
		\label{eq:qj-equation}
		(q_j)_t+\mu_jq_j=p,
		\qquad
		q_j(0)=0, \text{ and } \int_0^t\mathfrak M(t-r)p(r)\,dr
		=
		\sum_{j=1}^{J}\alpha_jq_j(t).
	\end{equation}
\end{lemma}

\begin{proof}
	Differentiating \eqref{eq:qj-definition} gives
	$
	(q_j)_t(t,x)=p(t,x)-\mu_jq_j(t,x),
	$
	and $q_j(0,x)=0$. Moreover,
	$
	\sum_{j=1}^J\alpha_jq_j(t,x)
	=
	\sum_{j=1}^J\alpha_j
	\int_0^t e^{-\mu_j(t-r)}p(r,x)\,dr. 
	$ Hence \eqref{eq:qj-equation} holds.

\end{proof}

\begin{remark}
	The auxiliary variables $q_j$ realize only the kernel $\mathfrak M$
	appearing in the state equation. No auxiliary realization of
	$\widetilde{\mathfrak M}$ is used or needed in the Carleman argument.
	Thus, $\mathfrak M$ and $\widetilde{\mathfrak M}$ may be different
	kernels.
\end{remark}

\begin{remark}
	\label{rem:3.8}
	The finite exponential representation \eqref{eq:memory-kernel-exp}
	converts the Volterra memory term into a finite family of auxiliary
	ODE variables. Thus, after time reversal, the adjoint equation is
	treated as a coupled parabolic--ODE system. The coupling between
	the parabolic component and the memory variables is lower order,
	whereas the delayed term is handled separately through the
	lag-gap condition \eqref{eq:multiplicative-lag-gap}.
\end{remark}

\subsection{Carleman estimate for the delayed--memory system}
\begin{lemma}[Weighted Volterra estimate for the exponential memory]
	\label{lemma:weighted-volterra-memory}
	Let $\mu>0$ and let
	$
	q(t,x)=\int_0^t e^{-\mu(t-r)}p(r,x)\,dr,
	 (t,x)\in Q.
	$
	Define
	$
	W_{s,\lambda}(t,x)
	:=
	(\lambda s\xi(t,x))^2e^{-2s\phi(t,x)}.
	$
	Then, for $\lambda\geq\lambda_0$ and
	$s\geq s_0(\lambda)$ sufficiently large,
	\[
	\iint_Q W_{s,\lambda}|q|^2\,dx\,dt
	\leq
	C_{s,\lambda}\mathcal R_{\rm sh}(p)
	+
	\frac{C}{s\lambda^2}I_H(v),
	\]
	where $C$ is independent of $p$.
\end{lemma}

\begin{proof}
	We first establish the monotonicity property of the weighted
	coefficient. From the assumptions \ref{assumption:carleman-weight} we have
	$
	W_{s,\lambda}(t,x)
	=
	(\lambda s)^2g(t)^2e^{2\lambda\rho(t,x)}
	e^{-2sg(t)\vartheta(t,x)}.
	$
	Therefore
	\begin{align}
		\partial_t\log W_{s,\lambda}
		={}&
		2\frac{g'(t)}{g(t)}
		+2\lambda\rho_t(t,x)
		-
		2s\bigl(g'(t)\vartheta(t,x)
		+g(t)\vartheta_t(t,x)\bigr).
		\label{eq:weight-log-derivative}
	\end{align}
	
	By the properties of the weight functions fixed in the preceding
	Carleman construction, the quantities involving
	$\rho_t$ and $\vartheta_t$ are uniformly controlled on the relevant
	time intervals. Since $\vartheta\geq c_0>0$, we obtain
	$
	g'(t)\vartheta(t,x)+g(t)\vartheta_t(t,x)
	\geq
	g(t)\left(
	\frac{\kappa c_0}{(T-t)^2}-C_1
	\right).
	$
	Choosing $\kappa$ sufficiently large gives
	\[
	g'(t)\vartheta(t,x)+g(t)\vartheta_t(t,x)
	\geq
	\frac{c_0}{2}g'(t).
	\]
	Since
	$
	\frac{g'(t)}{g(t)}
	=
	\frac{\kappa}{(T-t)^2},
	$
	we may then choose $s_0(\lambda)$ sufficiently large so that
	\[
	\partial_t\log W_{s,\lambda}(t,x)\leq0
	\qquad\text{in }Q.
	\]
	Consequently,
	\begin{equation}
		\label{eq:weight-monotonicity-volterra}
		W_{s,\lambda}(t,x)
		\leq
		W_{s,\lambda}(r,x),
		\qquad 0\leq r\leq t<T.
	\end{equation}
	
	On the other hand, by the Volterra representation and
	Cauchy--Schwarz,
	\begin{align*}
		|q(t,x)|^2
		&\leq
		\left(\int_0^t e^{-\mu(t-r)}\,dr\right)
		\left(
		\int_0^t e^{-\mu(t-r)}|p(r,x)|^2\,dr
		\right)
		\leq
		\frac1\mu
		\int_0^t e^{-\mu(t-r)}|p(r,x)|^2\,dr.
	\end{align*}
	Multiplying by $W_{s,\lambda}(t,x)$ and using
	\eqref{eq:weight-monotonicity-volterra}, we find
	\[
	W_{s,\lambda}(t,x)|q(t,x)|^2
	\leq
	\frac1\mu
	\int_0^t e^{-\mu(t-r)}
	W_{s,\lambda}(r,x)|p(r,x)|^2\,dr.
	\]
	Integration over $Q$ and Fubini's theorem yield
	\begin{align}
		\iint_QW_{s,\lambda}|q|^2\,dx\,dt
		&\leq
		\frac1\mu
		\int_\Theta\int_0^T
		\int_r^T
		e^{-\mu(t-r)}
		W_{s,\lambda}(r,x)|p(r,x)|^2
		\,dt\,dr\,dx
		\nonumber\\
		&\leq
		\frac1{\mu^2}
		\iint_QW_{s,\lambda}|p|^2\,dx\,dt.
		\label{eq:volterra-global}
	\end{align}
	
	We split the last integral into
	$(0,h+2\delta)$ and $(h+2\delta,T)$. Since
	$h+2\delta<T$, the weight is bounded on the first interval and hence
	$
	\int_0^{h+2\delta}\int_\Theta
	W_{s,\lambda}|p|^2\,dx\,dt
	\leq
	C_{s,\lambda}\mathcal R_{\rm sh}(p).
	$
	On $(h+2\delta,T)$ we have $\zeta=1$, and hence $v=p$.
	Moreover, since $\xi$ is bounded below by a positive constant,
	$
	(\lambda s\xi)^2
	\leq
	\frac{C}{s\lambda^2}
	s^3\lambda^4\xi^3.
	$
	Therefore
	\[
	\int_{h+2\delta}^{T}\int_\Theta
	W_{s,\lambda}|p|^2\,dx\,dt
	\leq
	\frac{C}{s\lambda^2}I_H(v).
	\]
	Combining these estimates with \eqref{eq:volterra-global} proves the
	claim.
\end{proof}
\begin{proposition}[Carleman estimate for the delayed--memory system]
	\label{prop:carleman-delay-memory}
	Assume \eqref{eq:memory-kernel-exp} and the hypotheses of
	Lemmas~\ref{lemma:carleman-parabolic-delay} and
	\ref{lemma:carleman-ODE}. Let $(p,q_1,\ldots,q_J)$ satisfy
	\begin{equation}
		\label{eq:coupled-delay-memory}
		\left\{
		\begin{aligned}
			p_t-\Delta p-\mathfrak A_1^*p(t-h)
			-\sum_{j=1}^{J}\alpha_jq_j
			&=F
			&&\text{in }Q,\\
			(q_j)_t+\mu_jq_j
			&=p
			&&\text{in }Q,\\
			q_j(0)&=0.
		\end{aligned}
		\right.
	\end{equation}
	Let $v=\zeta p$, where $\zeta$ is given by
	\eqref{eq:temporal-cutoff}. Then there exist
	$\lambda_0>0$ and $s_0(\lambda)>0$ such that, for
	$\lambda\geq\lambda_0$ and $s\geq s_0(\lambda)$,
	\begin{equation}
		\label{eq:coupled-carl-final}
		\begin{aligned}
			I_H(v)+\sum_{j=1}^{J}I_O(q_j)
			\leq{}&
			C\iint_Q|F|^2e^{-2s\phi}\,dx\,dt
			\\
			&+
			C\iint_{(0,T)\times\varpi(t)}
			(\lambda s\xi)^{m_0}
			|p|^2e^{-2s\phi}\,dx\,dt
			+
			C\mathcal R_{\rm sh}(p),
		\end{aligned}
	\end{equation}
	where $m_0\geq4$ is a fixed integer.
\end{proposition}

\begin{proof}
	We divide the proof into three steps.
	
	\textbf{Step 1.}
	From \eqref{eq:coupled-delay-memory} we have,
	$
	p_t-\Delta p-\mathfrak A_1^*p(t-h)
	=
	F+\sum_{j=1}^{J}\alpha_jq_j.
	$
	Applying Lemma~\ref{lemma:carleman-parabolic-delay}, we obtain
	\begin{align}
		I_H(v)
		\leq{}&
		C\iint_Q
		\left|F+\sum_{j=1}^{J}\alpha_jq_j\right|^2
		e^{-2s\phi}\,dx\,dt
		\nonumber\\
		&+
		C\iint_{(0,T)\times\varpi(t)}
		\lambda(\lambda s\xi)^3
		|p|^2e^{-2s\phi}\,dx\,dt
		+
		C\mathcal R_{\rm sh}(p).
		\label{eq:IH-first-correct}
	\end{align}
	The zero-order term in $I_O(q_j)$ gives
	$
	\iint_Q|q_j|^2e^{-2s\phi}\,dx\,dt
	\leq
	\frac{C}{\lambda^2s^2}I_O(q_j).
	$
	Hence
	\begin{align}
		I_H(v)
		\leq{}&
		C\iint_Q|F|^2e^{-2s\phi}\,dx\,dt
		+
		\frac{C}{\lambda^2s^2}
		\sum_{j=1}^{J}I_O(q_j)
		\nonumber\\
		&+
		C\iint_{(0,T)\times\varpi(t)}
		\lambda(\lambda s\xi)^3
		|p|^2e^{-2s\phi}\,dx\,dt
		+
		C\mathcal R_{\rm sh}(p).
		\label{eq:IH-with-IO-correct}
	\end{align}
	
	\medskip
	\noindent
	\textbf{Step 2.}
	For every $j=1,\ldots,J$, Lemma~\ref{lemma:carleman-ODE}
	gives
	\begin{equation}
		\label{eq:IO-before-localization}
		I_O(q_j)
		\leq
		C\iint_Q|p|^2e^{-2s\phi}\,dx\,dt
		+
		C\iint_{(0,T)\times\varpi(t)}
		(\lambda s\xi)^2|q_j|^2e^{-2s\phi}\,dx\,dt.
	\end{equation}
	
	By \eqref{eq:p-weighted-by-cutoff},
	\begin{equation}
		\label{eq:p-global-small-correct}
		\iint_Q|p|^2e^{-2s\phi}\,dx\,dt
		\leq
		C\mathcal R_{\rm sh}(p)
		+
		\frac{C}{s^3\lambda^4}I_H(v).
	\end{equation}
	
	Because $\varpi(t)\subset\Theta$,
	\[
	\iint_{(0,T)\times\varpi(t)}
	(\lambda s\xi)^2|q_j|^2e^{-2s\phi}\,dx\,dt
	\leq
	\iint_Q
	(\lambda s\xi)^2|q_j|^2e^{-2s\phi}\,dx\,dt.
	\]
	By Lemma~\ref{lemma:weighted-volterra-memory},
	\begin{equation}
		\label{eq:local-qj-closed}
		\iint_{(0,T)\times\varpi(t)}
		(\lambda s\xi)^2|q_j|^2e^{-2s\phi}\,dx\,dt
		\leq
		C\mathcal R_{\rm sh}(p)
		+
		\frac{C}{s\lambda^2}I_H(v).
	\end{equation}
	
	Substitution of \eqref{eq:p-global-small-correct} and
	\eqref{eq:local-qj-closed} into \eqref{eq:IO-before-localization}
	yields
	\begin{equation}
		\label{eq:IO-final-single}
		I_O(q_j)
		\leq
		\left(
		\frac{C}{s^3\lambda^4}
		+
		\frac{C}{s\lambda^2}
		\right)I_H(v)
		+
		C\mathcal R_{\rm sh}(p).
	\end{equation}
	Summing over $j$ gives
	\begin{equation}
		\label{eq:IO-final-sum}
		\sum_{j=1}^{J}I_O(q_j)
		\leq
		\left(
		\frac{C}{s^3\lambda^4}
		+
		\frac{C}{s\lambda^2}
		\right)I_H(v)
		+
		C\mathcal R_{\rm sh}(p).
	\end{equation}

	\textbf{Step 3.}
	Substituting \eqref{eq:IO-final-sum} into
	\eqref{eq:IH-with-IO-correct}, we obtain
	\begin{align}
		I_H(v)
		\leq{}&
		C\iint_Q|F|^2e^{-2s\phi}\,dx\,dt
		+
		C\left(
		\frac{1}{\lambda^2s^2}
		\right)
		\left(
		\frac{1}{s^3\lambda^4}
		+
		\frac{1}{s\lambda^2}
		\right)I_H(v)
		\nonumber\\
		&+
		C\iint_{(0,T)\times\varpi(t)}
		\lambda(\lambda s\xi)^3
		|p|^2e^{-2s\phi}\,dx\,dt
		+
		C\mathcal R_{\rm sh}(p).
		\label{eq:IH-preabsorption}
	\end{align}
	
	First fix $\lambda\geq\lambda_0$. We then choose
	$s\geq s_0(\lambda)$ sufficiently large so that
	$
	C
	\left(
	\frac{1}{\lambda^2s^2}
	\right)
	\left(
	\frac{1}{s^3\lambda^4}
	+
	\frac{1}{s\lambda^2}
	\right)
	\leq\frac14.
	$
	The corresponding term containing $I_H(v)$ can therefore be
	absorbed into the left-hand side. Thus
	\begin{align}
		I_H(v)
		\leq{}&
		C\iint_Q|F|^2e^{-2s\phi}\,dx\,dt
		+
		C\iint_{(0,T)\times\varpi(t)}
		\lambda(\lambda s\xi)^3
		|p|^2e^{-2s\phi}\,dx\,dt
		+
		C\mathcal R_{\rm sh}(p).
		\label{eq:IH-final}
	\end{align}
	
	Since $\lambda$ is fixed and $s\geq1$, there exists
	$m_0\geq4$ such that
	$
	\lambda(\lambda s\xi)^3
	\leq
	C(\lambda s\xi)^{m_0}.
	$
	Hence
	\begin{align}
		I_H(v)
		\leq{}&
		C\iint_Q|F|^2e^{-2s\phi}\,dx\,dt
		+
		C\iint_{(0,T)\times\varpi(t)}
		(\lambda s\xi)^{m_0}|p|^2e^{-2s\phi}\,dx\,dt
		+
		C\mathcal R_{\rm sh}(p).
		\label{eq:IH-final-power}
	\end{align}
	
	Finally, \eqref{eq:IO-final-sum} and
	\eqref{eq:IH-final-power} give
	\[
	I_H(v)+\sum_{j=1}^{J}I_O(q_j)
	\leq
	C\iint_Q|F|^2e^{-2s\phi}\,dx\,dt
	+
	C\iint_{(0,T)\times\varpi(t)}
	(\lambda s\xi)^{m_0}|p|^2e^{-2s\phi}\,dx\,dt
	+
	C\mathcal R_{\rm sh}(p).
	\]
	This proves \eqref{eq:coupled-carl-final}.

\end{proof}

\subsection{Preliminary observability estimate}

Let $w$ be a weak solution of the adjoint system
\begin{equation}
	\label{eq:adjoint-system}
	\left\{
	\begin{aligned}
		-w_t(t)
		={}&
		\Delta w(t)
		+\mathfrak A_1^*w(t+h)
		+\int_t^T
		\mathfrak M^*(s-t)w(s)\,ds
		\\
		&\quad
		-\widetilde{\mathfrak M}^*(T-t)z_T,
		&&0<t<T,
		\\
		w&=0,
		&&\text{on }(0,T)\times\partial\Theta,
		\\
		w(T)&=w_T,
		\\
		w(t)&=\psi(t),
		&&T<t\leq T+h.
	\end{aligned}
	\right.
\end{equation}

Introduce the time-reversed variable
$
p(t,x)=w(T-t,x).
$
Then
$
p(0)=w_T,
p(T)=w(0).
$
The corresponding observation region is $\widehat{\varpi}(t):=\varpi(T-t).$
Consequently,
\begin{equation*}
	\label{eq:observation-time-reversal}
	\int_0^T\|\mathfrak B^*w(t)\|_U^2\,dt
	=
	\int_0^T\int_{\widehat{\varpi}(t)}
	|p(t,x)|^2\,dx\,dt.
\end{equation*}
For $0<t<h$, the delayed value is supplied by the terminal extension,
namely
$
p(t-h)=\psi(T-t+h).
$
For the exponential kernel, define
\begin{equation*}
	\label{eq:qj-time-reversal}
	q_j(t,x)
	=
	\int_0^t
	e^{-\mu_j(t-r)}p(r,x)\,dr.
\end{equation*}
Then
\begin{equation*}
	\label{eq:qj-time-reversal-equation}
	(q_j)_t+\mu_jq_j=p,\quad
	q_j(0)=0, \text{ and }\int_t^T
	e^{-\mu_j(s-t)}w(s)\,ds
	=
	q_j(T-t).
\end{equation*}
Thus the finite exponential realization transforms the
time-reversed adjoint equation into
\begin{equation}
	\label{eq:time-reversed-adjoint}
	\left\{
	\begin{aligned}
		p_t-\Delta p-\mathfrak A_1^*p(t-h)
		-\sum_{j=1}^{J}\alpha_jq_j
		&=F(t),
		&&\text{in }Q,
		\\
		(q_j)_t+\mu_jq_j
		&=p,
		&&\text{in }Q,
		\\
		q_j(0)&=0.
	\end{aligned}
	\right.
\end{equation}
where
\begin{equation}
	\label{eq:F-terminal-memory}
	F(t)
	=
	-\widetilde{\mathfrak M}^{*}(t)z_T.
\end{equation}

Define
\begin{equation}
	\label{eq:T-eff-definition}
	\mathcal T_{\rm eff}(w_T,z_T)
	:=
	\|w_T\|_{L^2(\Theta)}^2
	+
	\int_0^T
	\|\widetilde{\mathfrak M}^{*}(t)z_T\|_{L^2(\Theta)}^2\,dt.
\end{equation}

\begin{assumption}[Short-time energy closure]
	\label{assumption:short-time-closure}
	There exists $C_{\rm sh}>0$, independent of the admissible terminal
	data $(w_T,\psi,z_T)$, such that the corresponding time-reversed
	adjoint solution satisfies
	\begin{equation}
		\label{eq:Rsh-preliminary}
		\mathcal R_{\rm sh}(p)
		\leq
		C_{\rm sh}
		\left[
		\mathcal T_{\rm eff}(w_T,z_T)
		+
		\int_0^T
		\|\mathfrak B^*w(t)\|_U^2\,dt
		\right].
	\end{equation}
\end{assumption}

\begin{theorem}[Preliminary observability estimate]
	\label{theorem:preliminary-observability}
	Under the hypotheses of
	Proposition~\ref{prop:carleman-delay-memory} and
	Assumption~\ref{assumption:short-time-closure}, there exists
	$C>0$ such that every admissible adjoint solution satisfies
	\begin{equation}
		\label{eq:preliminary-observability-final}
		\begin{aligned}
			\|w(0)\|_{L^2(\Theta)}^2
			+
			\int_0^h
			\|\mathfrak A_1^*w(t)\|_{L^2(\Theta)}^2\,dt
			\leq{}&
			C\int_0^T
			\|\mathfrak B^*w(t)\|_U^2\,dt
			\\
			&+
			C\mathcal T_{\rm eff}(w_T,z_T).
		\end{aligned}
	\end{equation}
\end{theorem}

\begin{proof}
	From \eqref{eq:F-terminal-memory},
	\[
	\|F\|_{L^2(Q)}^2
	=
	\int_0^T
	\|\widetilde{\mathfrak M}^{*}(t)z_T\|_{L^2(\Theta)}^2\,dt
	\leq
	\mathcal T_{\rm eff}(w_T,z_T).
	\]
	Applying Proposition~\ref{prop:carleman-delay-memory} to
	\eqref{eq:time-reversed-adjoint}, we obtain
	\begin{align}
		I_H(v)+\sum_{j=1}^{J}I_O(q_j)
		\leq{}&
		C\mathcal T_{\rm eff}(w_T,z_T)
		\nonumber\\
		&+
		C\iint_{(0,T)\times\widehat{\varpi}(t)}
		(\lambda s\xi)^{m_0}
		|p|^2e^{-2s\phi}\,dx\,dt
		+
		C\mathcal R_{\rm sh}(p).
		\label{eq:preliminary-carl}
	\end{align}
	
	For fixed $\lambda$ and $s$, the function
	$
	(\lambda s\xi)^{m_0}e^{-2s\phi}
	$
	is bounded on $Q$. Indeed, from the definitions of $\xi$ and
	$\phi$, there exist positive constants $C_1,C_2$ such that
	$
	\xi(t,x)\leq C_1g(t),
	\quad
	\phi(t,x)\geq C_2g(t).
	$
	Therefore
	\[
	(\lambda s\xi)^{m_0}e^{-2s\phi}
	\leq
	C(\lambda s)^{m_0}
	g(t)^{m_0}e^{-2sC_2g(t)},
	\]
	and the expression on the right is bounded because
	$r^{m_0}e^{-2sC_2r}$ is bounded for $r\geq1$.
	
	Hence
	\begin{align}
		\iint_{(0,T)\times\widehat{\varpi}(t)}
		(\lambda s\xi)^{m_0}|p|^2e^{-2s\phi}\,dx\,dt
		&\leq
		C_{s,\lambda}
		\int_0^T\int_{\widehat{\varpi}(t)}
		|p(t,x)|^2\,dx\,dt
		\nonumber\\
		&=
		C_{s,\lambda}
		\int_0^T
		\|\mathfrak B^*w(t)\|_U^2\,dt.
		\label{eq:weighted-local-observation}
	\end{align}
	
	Next, since $\zeta=1$ on $[h+2\delta,T)$, we have
	$v=p$ on this interval. A standard parabolic energy estimate on
	the terminal interval, together with the boundedness of
	$\mathfrak A_1^*$ and the finite-dimensional representation of
	the memory term, yields
	\begin{equation}
		\label{eq:terminal-strip-full}
		\|p(T)\|_{L^2(\Theta)}^2
		+
		\int_{T-h}^{T}\|p(t)\|_{L^2(\Theta)}^2\,dt
		\leq
		C\left[
		I_H(v)
		+
		\sum_{j=1}^{J}I_O(q_j)
		+
		\|F\|_{L^2(Q)}^2
		+
		\mathcal R_{\rm sh}(p)
		\right].
	\end{equation}
	
	Since $p(T)=w(0)$ and
	$w(t)=p(T-t)$,
	\[
	\int_0^h
	\|\mathfrak A_1^*w(t)\|_{L^2(\Theta)}^2\,dt
	=
	\int_{T-h}^{T}
	\|\mathfrak A_1^*p(r)\|_{L^2(\Theta)}^2\,dr.
	\]
	Using the boundedness of $\mathfrak A_1^*$,
	\begin{equation}
		\label{eq:history-from-terminal-strip}
		\int_0^h
		\|\mathfrak A_1^*w(t)\|_{L^2(\Theta)}^2\,dt
		\leq
		K^2
		\int_{T-h}^{T}\|p(r)\|_{L^2(\Theta)}^2\,dr.
	\end{equation}
	Thus \eqref{eq:terminal-strip-full} gives
	\begin{align}
		\|w(0)\|_{L^2(\Theta)}^2
		+
		\int_0^h
		\|\mathfrak A_1^*w(t)\|_{L^2(\Theta)}^2\,dt
		\leq{}&
		C\left[
		I_H(v)+\sum_{j=1}^{J}I_O(q_j)
		\right]
		\nonumber\\
		&+
		C\mathcal T_{\rm eff}(w_T,z_T)
		+
		C\mathcal R_{\rm sh}(p).
		\label{eq:principal-from-carl}
	\end{align}
	
	Finally, we use Assumption~\ref{assumption:short-time-closure}.
	Combining \eqref{eq:preliminary-carl},
	\eqref{eq:weighted-local-observation}, and
	\eqref{eq:Rsh-preliminary} in
	\eqref{eq:principal-from-carl}, and recalling the bound for $F$,
	gives \eqref{eq:preliminary-observability-final}.
\end{proof}

\subsection{Compactness--uniqueness and elimination of the terminal
	remainder}

\begin{assumption}[Adjoint unique continuation]
	\label{assumption:adjoint-UCP}
	For every admissible adjoint solution associated with terminal data
	$(w_T,\psi,z_T)$,
	\begin{equation}
		\label{eq:section3-UCP}
		\mathfrak B^*w(t)=0
		\quad\text{for a.e. }t\in(0,T)
		\quad\Longrightarrow\quad
		w_T=0,
		\qquad
		\widetilde{\mathfrak M}^{*}(t)z_T=0
		\quad\text{for a.e. }t\in(0,T).
	\end{equation}
\end{assumption}

\begin{assumption}[Terminal compactness--uniqueness]
	\label{assumption:terminal-compactness}
	Let $\{w^k\}$ be a sequence of admissible adjoint solutions such that
	\[
	\sup_k
	\mathcal T_{\rm eff}(w_T^k,z_T^k)<\infty \text{ and }
	\int_0^T
	\|\mathfrak B^*w^k(t)\|_U^2\,dt
	\longrightarrow0.
	\]
	Then, after passage to a subsequence, the terminal quantities converge
	strongly enough that
	\[
	\mathcal T_{\rm eff}(w_T^k,z_T^k)
	\longrightarrow
	\mathcal T_{\rm eff}(w_T^*,z_T^*),
	\]
	where $(w_T^*,z_T^*)$ are the terminal data associated with the
	limiting adjoint trajectory.
\end{assumption}

\begin{lemma}[Compactness--uniqueness estimate]
	\label{lemma:compactness-uniqueness}
	Assume the hypotheses of
	Theorem~\ref{theorem:preliminary-observability},
	Assumption~\ref{assumption:adjoint-UCP}, and
	Assumption~\ref{assumption:terminal-compactness}. Then, for every
	$\varepsilon>0$, there exists $C_\varepsilon>0$ such that
	\begin{equation}
		\label{eq:compactness-terminal}
		\mathcal T_{\rm eff}(w_T,z_T)
		\leq
		\varepsilon
		\left[
		\|w(0)\|_{L^2(\Theta)}^2
		+
		\int_0^h
		\|\mathfrak A_1^*w(t)\|_{L^2(\Theta)}^2\,dt
		\right]
		+
		C_\varepsilon
		\int_0^T
		\|\mathfrak B^*w(t)\|_U^2\,dt.
	\end{equation}
\end{lemma}

\begin{proof}
	We argue by contradiction. Suppose that
	\eqref{eq:compactness-terminal} fails for some $\varepsilon_0>0$.
	Then, for every $k\in\mathbb N$, there exists an admissible adjoint
	solution $w^k$ such that
	\begin{align}
		\mathcal T_{\rm eff}(w_T^k,z_T^k)
		>{}&
		\varepsilon_0
		\left[
		\|w^k(0)\|_{L^2(\Theta)}^2
		+
		\int_0^h
		\|\mathfrak A_1^*w^k(t)\|_{L^2(\Theta)}^2\,dt
		\right]
		\nonumber\\
		&+
		k\int_0^T
		\|\mathfrak B^*w^k(t)\|_U^2\,dt.
		\label{eq:compactness-contradiction}
	\end{align}
	
	By homogeneity we normalize
	\begin{equation}
		\label{eq:compactness-normalization}
		\mathcal T_{\rm eff}(w_T^k,z_T^k)=1.
	\end{equation}
	Then
	\[
	\int_0^T
	\|\mathfrak B^*w^k(t)\|_U^2\,dt
	\longrightarrow0.
	\]
	
	The preliminary observability estimate gives a uniform bound for
	the corresponding adjoint trajectories in
	$
	L^\infty(0,T;L^2(\Theta))
	\cap
	L^2(0,T;H_0^1(\Theta)).
	$
	Moreover, for every $\eta>0$, the adjoint equation and the
	boundedness of the delay and memory operators imply
	\[
	\{w_t^k\}
	\quad\text{is bounded in}\quad
	L^2(0,T-\eta;H^{-1}(\Theta)).
	\]
	Since
	$
	H_0^1(\Theta)\Subset L^2(\Theta)
	\hookrightarrow H^{-1}(\Theta),
	$
	the Aubin--Lions lemma yields, after extraction of a subsequence,
	\[
	w^k\longrightarrow w^*
	\qquad\text{strongly in }
	L^2((0,T-\eta)\times\Theta)
	\]
	for every $\eta>0$.
	
	The memory variables inherit the same compactness from their finite
	convolution representation. Consequently,
	\[
	q_j^k\longrightarrow q_j^*
	\qquad\text{strongly on every }
	[0,T-\eta].
	\]
	
	From \eqref{eq:compactness-normalization},
	\[
	\mathfrak B^*w^k\longrightarrow0
	\qquad\text{in }L^2(0,T;U).
	\]
	Passing to the limit on every interval $(0,T-\eta)$ and then letting
	$\eta\downarrow0$ gives
	\[
	\mathfrak B^*w^*=0
	\qquad\text{for a.e. }t\in(0,T).
	\]
	By Assumption~\ref{assumption:adjoint-UCP},
	\[
	w_T^*=0,
	\qquad
	\widetilde{\mathfrak M}^{*}(t)z_T^*=0
	\quad\text{for a.e. }t\in(0,T).
	\]
	Therefore
	\[
	\mathcal T_{\rm eff}(w_T^*,z_T^*)=0.
	\]
	
	By Assumption~\ref{assumption:terminal-compactness},
	\[
	\mathcal T_{\rm eff}(w_T^k,z_T^k)
	\longrightarrow
	\mathcal T_{\rm eff}(w_T^*,z_T^*)=0.
	\]
	This contradicts the normalization
	\[
	\mathcal T_{\rm eff}(w_T^k,z_T^k)=1.
	\]
	Therefore \eqref{eq:compactness-terminal} holds.
\end{proof}

\begin{corollary}[Observability inequality]
	\label{cor:final-observability}
	Under the hypotheses of
	Theorem~\ref{theorem:preliminary-observability},
	Assumptions~\ref{assumption:adjoint-UCP} and
	\ref{assumption:terminal-compactness}, and
	Lemma~\ref{lemma:compactness-uniqueness}, there exists
	$C_{\rm obs}>0$ such that every admissible adjoint solution satisfies
	\begin{equation}
		\label{eq:observability-A1}
		\|w(0)\|_{L^2(\Theta)}^2
		+
		\int_0^h
		\|\mathfrak A_1^*w(t)\|_{L^2(\Theta)}^2\,dt
		\leq
		C_{\rm obs}
		\int_0^T
		\|\mathfrak B^*w(t)\|_U^2\,dt.
	\end{equation}
\end{corollary}

\begin{proof}
	By Theorem~\ref{theorem:preliminary-observability},
	\begin{align}
		\|w(0)\|_{L^2(\Theta)}^2
		+
		\int_0^h
		\|\mathfrak A_1^*w(t)\|_{L^2(\Theta)}^2\,dt
		\leq{}&
		C\int_0^T
		\|\mathfrak B^*w(t)\|_U^2\,dt
		\nonumber\\
		&+
		C\mathcal T_{\rm eff}(w_T,z_T).
		\label{eq:obs-terminal-combination}
	\end{align}
	
	By Lemma~\ref{lemma:compactness-uniqueness}, for every
	$\varepsilon>0$,
	\begin{align}
		\mathcal T_{\rm eff}(w_T,z_T)
		\leq{}&
		\varepsilon
		\left[
		\|w(0)\|_{L^2(\Theta)}^2
		+
		\int_0^h
		\|\mathfrak A_1^*w(t)\|_{L^2(\Theta)}^2\,dt
		\right]
		\nonumber\\
		&+
		C_\varepsilon
		\int_0^T
		\|\mathfrak B^*w(t)\|_U^2\,dt.
		\label{eq:terminal-small}
	\end{align}
	
	Choose $\varepsilon>0$ so small that $C\varepsilon<1/2$.
	Substitution of \eqref{eq:terminal-small} into
	\eqref{eq:obs-terminal-combination} gives
	\begin{align*}
		\frac12
		\left[
		\|w(0)\|_{L^2(\Theta)}^2
		+
		\int_0^h
		\|\mathfrak A_1^*w(t)\|_{L^2(\Theta)}^2\,dt
		\right]
		\leq
		C
		\int_0^T
		\|\mathfrak B^*w(t)\|_U^2\,dt.
	\end{align*}
	Absorbing the factor $1/2$ into the constant proves
	\eqref{eq:observability-A1}.
\end{proof}

\subsection{Null controllability}

\begin{theorem}[Delay and memory-type null controllability]
	\label{thm:null-controllability-section3}
	Assume the hypotheses of Corollary~\ref{cor:final-observability}.
	Then, for every initial datum
	$\phi\in\mathcal H_h$, there exists
	$u\in L^2(0,T;U)$ such that the corresponding solution $y$ of
	\eqref{eq:60} satisfies
	\begin{equation*}
		y(T)=0,\quad \mathfrak A_1y(t)=0,
		\quad
		\text{for a.e. }t\in[T-h,T),\text{ and } \int_0^T
		\widetilde{\mathfrak M}(T-s)y(s)\,ds=0.
	\end{equation*}
	Thus \eqref{eq:60} is delay and memory-type null controllable in the
	sense of Definition~\ref{def:Def_memory_type of null control}.
\end{theorem}

\begin{proof}
	The observability inequality \eqref{eq:observability-A1} is
	precisely the observability property required in
	Theorem~\ref{prop:2.1}. Hence, by the abstract duality result
	established in Section~\ref{sec:Abstract Duality and Observability
		Analogue}, there exists
	$u\in L^2(0,T;U)$ such that the corresponding solution of
	\eqref{eq:60} satisfies
	\begin{equation*}
		y(T)=0,
		\qquad
		\mathfrak A_1y(t)=0
		\quad\text{for a.e. }t\in[T-h,T),\text{ and } \int_0^T
		\widetilde{\mathfrak M}(T-s)y(s)\,ds=0.
	\end{equation*}
	These are exactly the terminal conditions in
	Definition~\ref{def:Def_memory_type of null control}. Therefore,
	\eqref{eq:60} is Delay and Memory--Type Null Controllable.
\end{proof}

\section{Numerical Experiments}
\label{sec:Application}

This section presents numerical experiments illustrating the constructive
Hilbert Uniqueness Method (HUM) for the delayed parabolic system with
memory considered in the previous sections. We restrict the numerical
study to the bounded-delay case $\mathfrak A_1=I$ and use exponential
memory kernels. The control is localized on a time-dependent moving
region. The main objectives are to describe the discrete HUM algorithm,
compare moving and fixed control configurations, and examine mesh and
parameter sensitivity.

The computations are intended as a numerical illustration of the
analytical framework. In particular, they do not address unbounded
delayed operators such as a delayed Laplacian.

\subsection{Model and constructive HUM control}
\label{subsec:num-setup}

We consider
$
Y=U=L^2(0,\pi)
$
and
$
\mathfrak A=\partial_{xx}
$
with
$
D(\mathfrak A)=H^2(0,\pi)\cap H_0^1(0,\pi).
$
The eigenpairs are
$
e_n(x)=\sqrt{2/\pi}\sin(nx)
$
and
$
\nu_n=-n^2$, $n\geq1$.
Hence the semigroup generated by $\mathfrak A$ is analytic and compact
for every positive time.

The control operator is
$
\mathfrak B(t)u=\chi_{\varpi(t)}u
$
and therefore
$
\mathfrak B^*(t)w=\chi_{\varpi(t)}w.
$
We take
$
\mathfrak A_1=I
$
and, for the numerical experiment, choose $\mathfrak M(t)=e^{-t},
\widetilde{\mathfrak M}(t)=e^{-t}.$
Thus, although the analytical definition allows
$\widetilde{\mathfrak M}$ to be different from $\mathfrak M$, the
numerical experiment takes them equal.
The numerical state equation is
\begin{equation}
	\label{eq:num-model}
	\begin{cases}
		y_t-y_{xx}-y(t-h,x)
		+\displaystyle\int_0^t e^{-(t-s)}y(s,x)\,ds
		=\chi_{\varpi(t)}(x)u(t,x),\\
		y(t,0)=y(t,\pi)=0,\\
		y(t,x)=\phi(t,x),\qquad t\in[-h,0].
	\end{cases}
\end{equation}
We use
$
h=0.1$, $T=1
$
and
$
\phi(t,x)=\sin x.
$
According to Definition~\ref{def:Def_memory_type of null control}, the
terminal conditions to be checked numerically are
\begin{equation}
	\label{eq:num-terminal-conditions}
	y(T)=0,
	\qquad
	y(t)=0\quad\text{for a.e. }t\in[T-h,T),
	\qquad
	\int_0^T e^{-(T-s)}y(s)\,ds=0,
\end{equation}
where the second condition follows from
$\mathfrak A_1=I$. Thus the terminal memory constraint is evaluated over
the full interval $[0,T]$.

\subsection{Discretization and memory realization}
\label{subsec:discretization}

We use backward Euler in time and centered finite differences in space.
For the reference computation,
$
M=201
$
time levels and
$
N=80
$
spatial points are used, with
$
\Delta t=T/(M-1)=1/200.
$
Since
$
h/\Delta t=20,
$
the delay is represented exactly on the time grid by
$
K=20.
$
To avoid a direct convolution at every time step, introduce the memory
variable $J(t,x)=\int_0^t e^{-(t-s)}y(s,x)\,ds.$
It satisfies $J_t+J=y, J(0,x)=0.$
Consequently, the system is written as
\begin{equation}
	\label{eq:num-realized-system}
	\begin{cases}
		y_t-y_{xx}-y(t-h)+J=\chi_{\varpi(t)}u,\\
		J_t+J=y.
	\end{cases}
\end{equation}

The backward Euler approximation of the memory equation is
\begin{equation}
	\label{eq:num-memory-discrete}
	J^n
	=
	\frac{J^{n-1}+\Delta t\,y^n}{1+\Delta t},
	\qquad J^0=0,
\end{equation}
while the delayed state is represented by
$
y^{n-K}
$
with the prescribed history used whenever $n<K$. The discrete forward
equation is therefore
\begin{equation}
	\label{eq:num-forward-discrete}
	\frac{y^n-y^{n-1}}{\Delta t}
	-\Delta_h y^n-y^{n-K}+J^n
	=
	\chi_{\varpi(t_n)}u^n.
\end{equation}

Because $J(T)$ realizes the full exponential memory functional, the
discrete terminal memory component is taken directly as
\begin{equation}
	\label{eq:num-memory-terminal-discrete}
	\mathcal M_h(y)=J^{M-1}
	\approx
	\int_0^T e^{-(T-s)}y(s)\,ds.
\end{equation}
Thus no truncation at $T-h$ is introduced in the numerical
implementation.

\subsection{Discrete HUM formulation}
\label{subsec:algorithmic-hum}

Let $y[u]$ denote the discrete state generated by a control $u$, and
let $y_0=y[0]$ denote the uncontrolled trajectory. We define the
discrete terminal map
\begin{equation}
	\label{eq:num-terminal-map}
	G_h(y)
	=
	\left(
	y^{M-1},
	\mathcal M_h(y),
	\mathcal H_h(y)
	\right),
\end{equation}
where $\mathcal M_h$ is given by
\eqref{eq:num-memory-terminal-discrete} and $\mathcal H_h(y)
=
\{y^n\}_{\,M-1-K\leq n\leq M-2}.$
The history block represents the half-open interval
$[T-h,T)$, whereas the terminal value $y^{M-1}=y(T)$ is treated as a
separate component of $G_h$.
For a terminal dual variable $\mu$, let $W[\mu]$ denote the associated
discrete adjoint trajectory. The HUM control is synthesized according
to
\begin{equation}
	\label{eq:num-hum-control}
	u_\mu^n
	=
	-\mathfrak B(t_n)^*W[\mu]^n
	=
	-\chi_{\varpi(t_n)}W[\mu]^n.
\end{equation}
The discrete HUM operator is then evaluated through $\Lambda_h\mu
=
G_h\bigl(y[u_\mu]\bigr),$
where the forward computation starts from zero initial data. The
discrete HUM equation is $g=-G_h(y_0).$
Thus the numerical HUM construction seeks a control for which all three
components of the terminal map vanish: the terminal state, the terminal
history associated with the delay condition, and the full terminal
memory quantity.

\subsection{Matrix-free HUM algorithm}
\label{subsec:matrix-free}

The matrix representing $\Lambda_h$ is not assembled. Instead, its
action on a terminal dual vector is computed by one backward adjoint
solve followed by one forward state solve. The resulting procedure is
summarized in Algorithm~\ref{alg:matrix-free-hum}.

\refstepcounter{algorithm}
\label{alg:matrix-free-hum}

\noindent\textbf{Algorithm~\thealgorithm: Matrix-free HUM computation}
\begin{center}
	\small
	\fbox{
		\begin{minipage}{0.93\textwidth}
			\textbf{Step 1 (Backward adjoint solve):}
			Solve the discrete adjoint delay--memory system backward from
			$T$ and compute $W[\mu]$.
			
			\medskip
			
			\textbf{Step 2 (Control synthesis):}
			Set
			\[
			u_\mu^n=-\chi_{\varpi(t_n)}W[\mu]^n.
			\]
			
			\medskip
			
			\textbf{Step 3 (Forward solve):}
			Solve the discrete forward system with zero initial datum and
			control $u_\mu$, obtaining $y[u_\mu]$.
			
			\medskip
			
			\textbf{Step 4 (Terminal extraction):}
			Evaluate the terminal value, full memory quantity, and terminal
			history block in $G_h$ and set
			\[
			\Lambda_h\mu=G_h(y[u_\mu]).
			\]
		\end{minipage}
	}
\end{center}
\medskip
\noindent

Thus, the dense terminal-to-terminal HUM matrix is never constructed.
Each MINRES matrix-vector product requires one backward adjoint solve,
one control-synthesis step, and one forward state solve. The factor
$\Delta t$ is incorporated consistently in the discrete terminal inner
products.

For the numerical comparisons below, we use fixed MINRES iteration
budgets rather than interpreting the iteration count as a continuum
convergence result.

\subsection{Moving and fixed control regions}
\label{subsec:matched}

We compare the moving region
\begin{equation}
	\label{eq:num-moving-region}
	\varpi(t)
	=
	\left(\delta(t),\delta(t)+\frac{\pi}{2}\right),
	\qquad
	\delta(t)=\frac{\pi}{2}\frac{t}{T},
\end{equation}
with the fixed region
\begin{equation}
	\label{eq:num-fixed-region}
	\varpi_{\rm fix}=(0.2\pi,0.7\pi).
\end{equation}

Both configurations are computed with the same spatial and temporal
discretization and the same MINRES iteration budget. In accordance with
the three terminal requirements in
\eqref{eq:num-terminal-conditions}, we use
\begin{equation}
	\label{eq:num-residuals-1}
	R_T=\|y(T)\|_{L^2(0,\pi)},
	\qquad
	R_M=
	\left\|
	\int_0^T e^{-(T-s)}y(s)\,ds
	\right\|_{L^2(0,\pi)},
\end{equation}
and
\begin{equation}
	\label{eq:num-residuals-2}
	R_H
	=
	\|y\|_{L^2(T-h,T;L^2(0,\pi))},
	\qquad
	C_u=\|u\|_{L^2(Q)}^2.
\end{equation}

Since $\mathfrak A_1=I$, $R_H$ measures the terminal history
condition $\mathfrak A_1y=0$ over $[T-h,T)$. The quantity $R_M$
corresponds exactly to the full terminal memory condition in
Definition~\ref{def:Def_memory_type of null control}.
\begin{table}[h!]
	\centering
	\caption{Comparison of the moving and fixed control configurations
		using $300$ MINRES iterations.}
	\label{tab:control-comparison}
	\small
	\begin{tabular}{lcccc}
		\toprule
		Configuration & $R_T$ & $R_M$ & $R_H$ & $C_u$\\
		\midrule
		Moving &
		$1.328161\times10^{-5}$ &
		$8.649643\times10^{-5}$ &
		$6.539804\times10^{-5}$ &
		$19.28939$\\
		Fixed &
		$1.032858\times10^{-3}$ &
		$3.347846\times10^{-3}$ &
		$2.887910\times10^{-3}$ &
		$2217.427$\\
		\bottomrule
	\end{tabular}
\end{table}
The moving configuration gives smaller terminal residuals and a
substantially smaller computed control cost under the same discrete
MINRES budget. The comparison concerns the two discrete HUM
computations and is not a claim that the fixed continuous system is
non-controllable.

\begin{figure}[h!]
	\centering
	\includegraphics[width=0.45\textwidth]
	{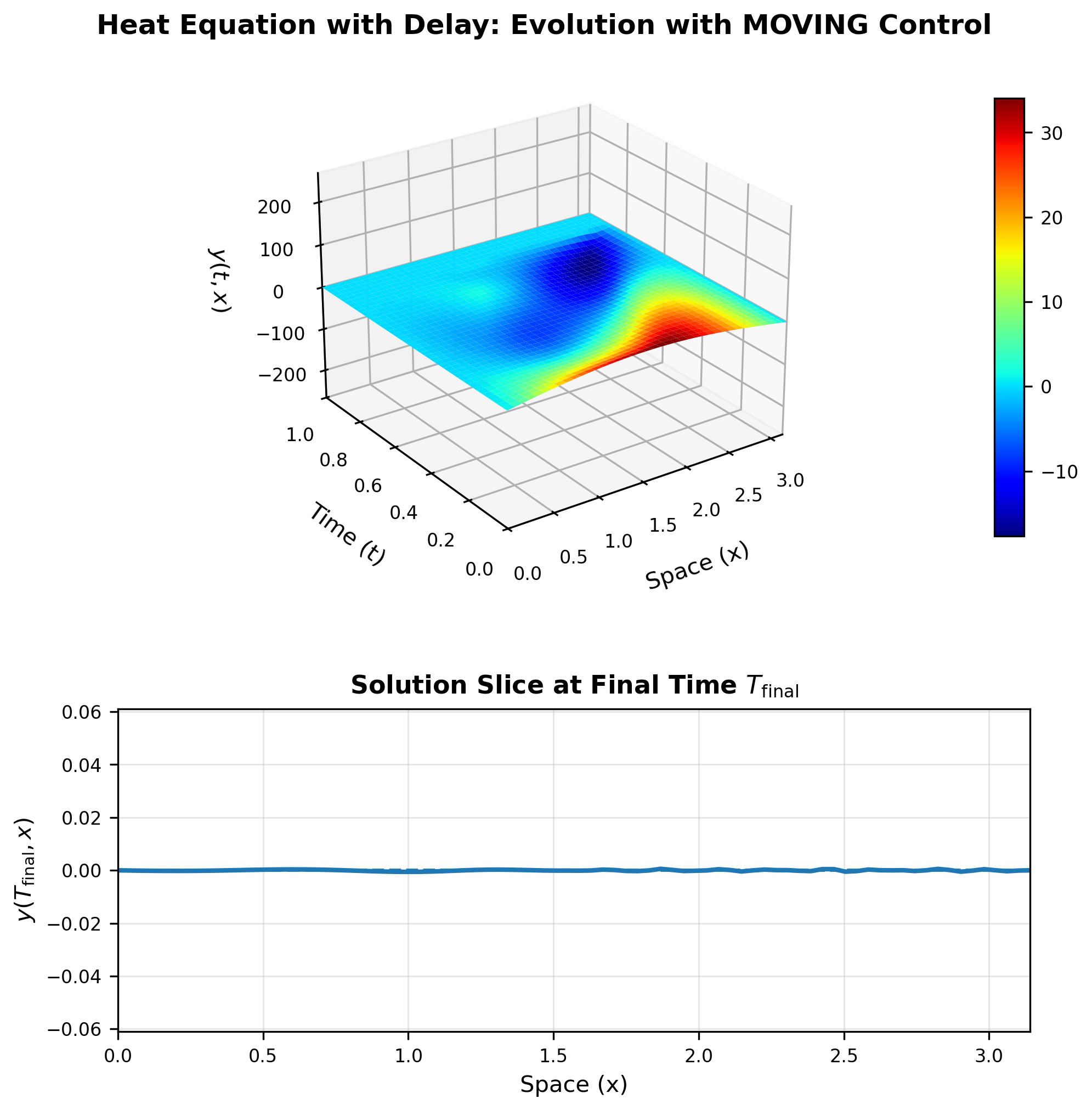}
	\includegraphics[width=0.45\textwidth]
	{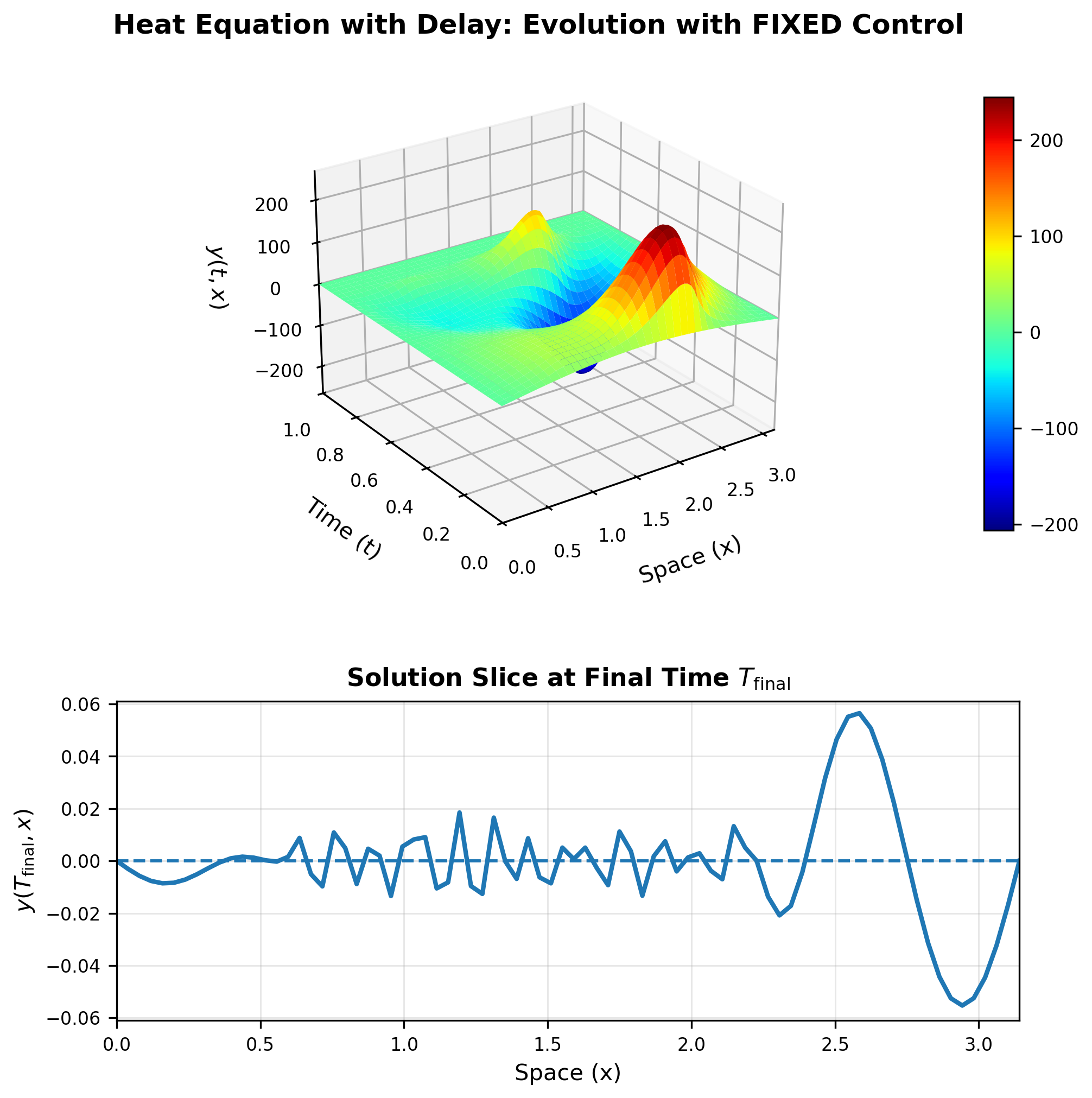}
	\caption{
		Spatiotemporal evolution $y(t,x)$ and terminal state profile $y(T,x)$ of the heat equation with delay and memory under moving (left) and fixed (right) control configurations. Under the moving control domain $\varpi(t)$, the state is effectively steered to zero across the spatial domain at $T=1$, whereas the fixed control region fails to dissipate state oscillations within the same 300 MINRES iterations (using HUM).
	}
	\label{fig:control_comparison}
\end{figure}

\subsection{Verification of the terminal conditions}
\label{subsec:verify}

The three quantities in
\eqref{eq:num-residuals-1}--\eqref{eq:num-residuals-2} provide a direct
numerical check of the three terminal requirements. Since
$\mathfrak A_1=I$, the history residual is the discrete $L^2$ norm of
the state on the terminal interval $[T-h,T)$.

The memory residual is computed from the terminal value $J(T)$ of the
memory realization and therefore corresponds to the full interval
$[0,T]$, in accordance with the revised definition.

\begin{figure}[h!]
	\centering
	\includegraphics[width=0.70\textwidth]
	{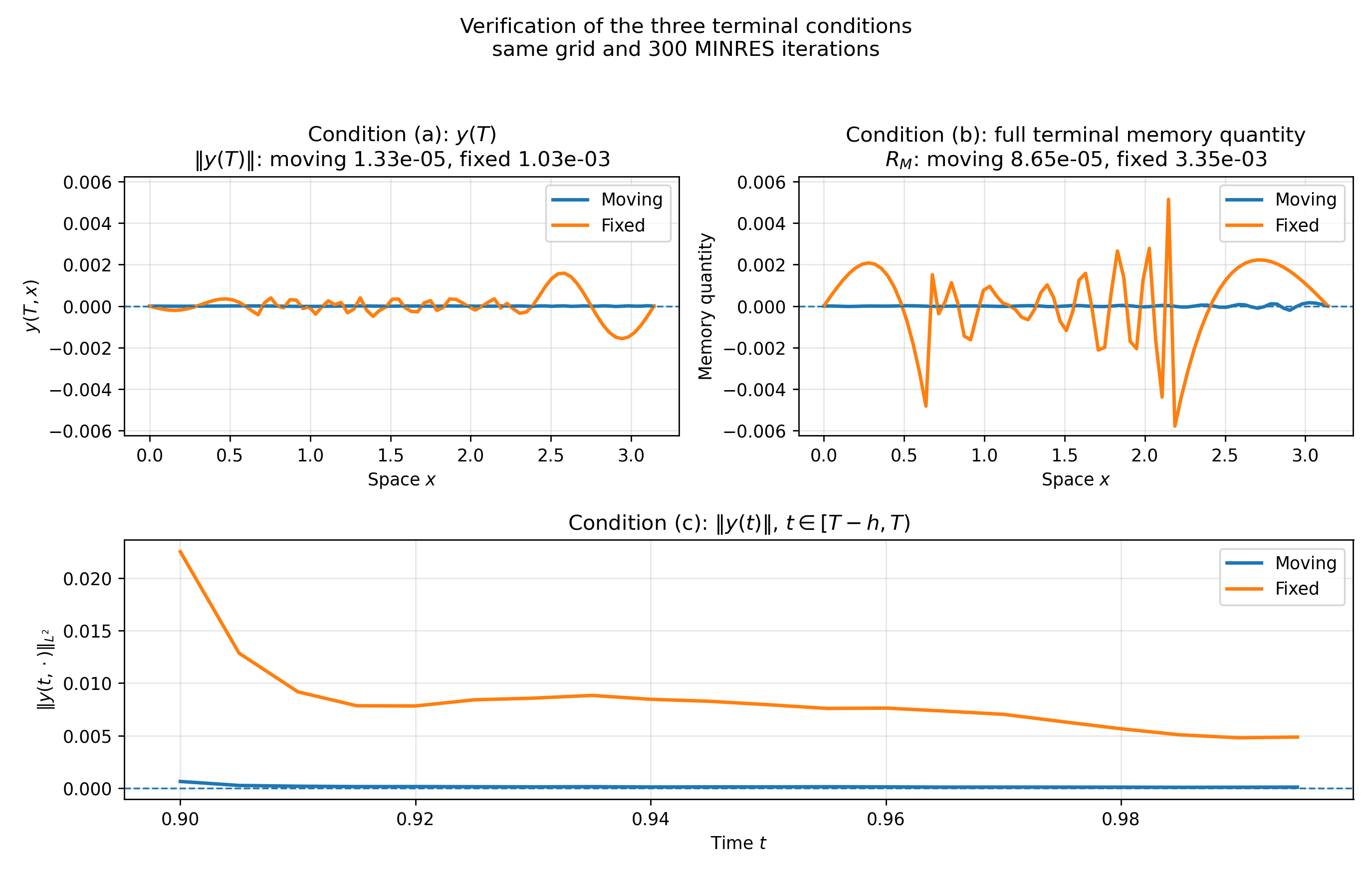}
	\caption{
		Verification of the three terminal conditions for delay and memory-type null controllability using 300 MINRES iterations: (a) final state profile $R_T$, (b) spatial distribution of the accumulated terminal memory functional $R_M$, and (c) temporal norm $R_H$ over the delay history interval $t \in [T-h, T)$. The moving control configuration achieves near-zero residuals across all three metrics compared to the fixed control setup. The corresponding residuals are reported in
		Table~\ref{tab:control-comparison}.
	}
	\label{fig:verification}
\end{figure}

Figure~\ref{fig:verification} provides a graphical representation of
the same three residuals reported in
Table~\ref{tab:control-comparison}; no additional numerical data are
introduced here.

\subsection{Cost--accuracy trade-off}
\label{subsec:tradeoff}

To examine the behavior of the iterative HUM construction, we perform
a $600$-iteration MINRES computation with $N=80$ and $M=201$. At each
iteration we record the HUM residual and the corresponding computed
control cost.

\begin{figure}[h!]
	\centering
	\includegraphics[width=0.50\textwidth]
	{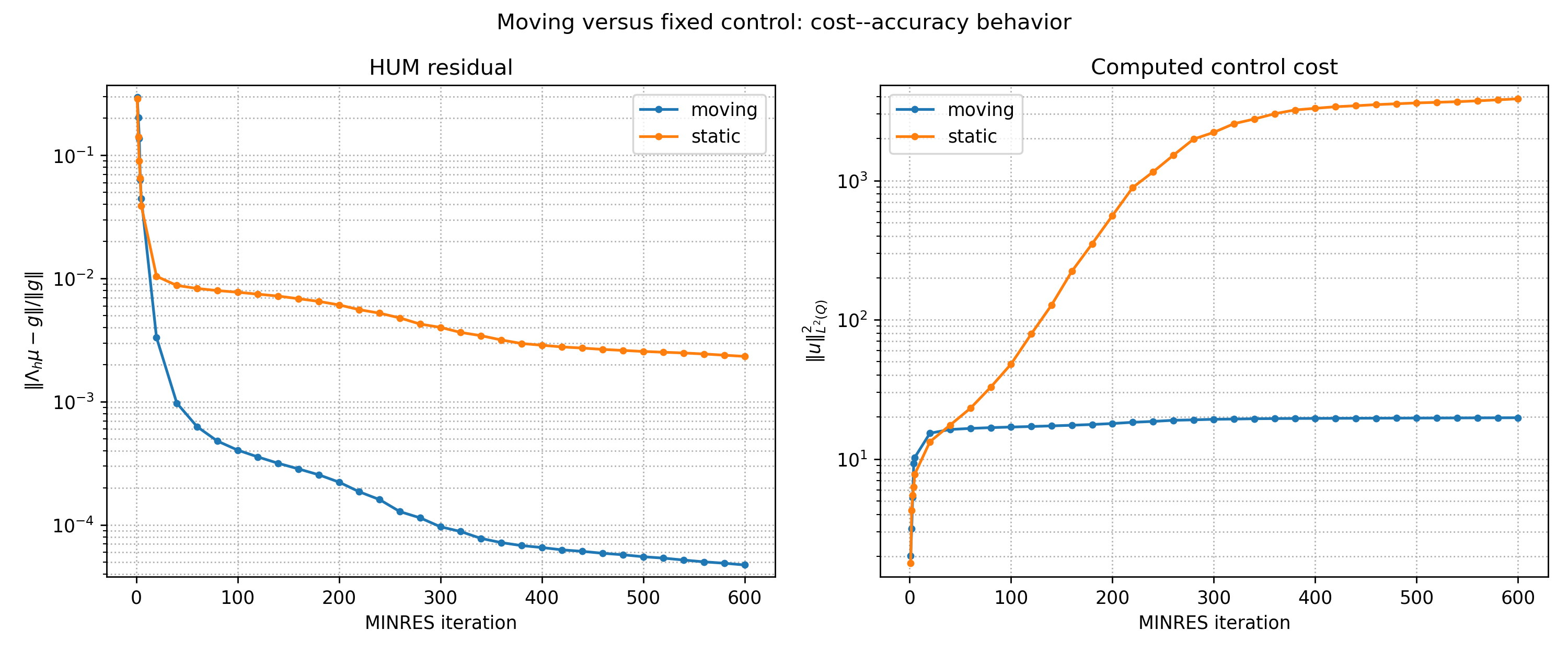}
	\caption{
Convergence behavior and cost--accuracy trade-off across 600 MINRES iterations. Left: Decay of the relative discrete HUM residual $\|\Lambda_h \mu - g\|/\|g\|$, represents the target vector of uncontrolled terminal constraints. Right: Evolution of the computed $L^2(Q)$ control cost $\|u\|_{L^2(Q)}^2$. Moving controls yield a faster residual reduction by several orders of magnitude while drastically reducing the required control effort.
	}
	\label{fig:tradeoff}
\end{figure}

At the first recorded and final iterations, the moving computation has
a residual decreasing from approximately $2.964\times10^{-1}$ to
$4.734\times10^{-5}$, while the fixed computation decreases from
approximately $2.896\times10^{-1}$ to $2.335\times10^{-3}$. Over the
same interval, the computed control costs increase from approximately
$2.011$ to $19.79$ for the moving configuration and from approximately
$1.780$ to $3.859\times10^3$ for the fixed configuration. Thus the
moving configuration exhibits substantially more favorable
cost--accuracy behavior for the tested discrete problem. These
observations concern the discrete systems and the specified MINRES
budget and should not be interpreted as a proof of a continuum
controllability obstruction.

\subsection{Mesh refinement}
\label{subsec:mesh-refinement}

We next refine the spatial and temporal grids simultaneously for the
moving control configuration, keeping $h=0.10$, $T=1$, and $300$
MINRES iterations.

\begin{table}[h!]
	\centering
	\caption{Mesh-refinement study for the moving control configuration
		using $300$ MINRES iterations.}
	\label{tab:mesh_refinement}
	\small
	\setlength{\tabcolsep}{4pt}
	\begin{tabular}{cccccc}
		\toprule
		$N$ & $M$ & $\Delta t$ & $R_T$ & $R_M$ & $C_u$\\
		\midrule
		40  & 101 & $0.01000$
		& $9.446\times10^{-6}$ & $8.904\times10^{-5}$ & $20.97191$\\
		80  & 201 & $0.00500$
		& $1.328\times10^{-5}$ & $8.650\times10^{-5}$ & $19.28939$\\
		160 & 401 & $0.00250$
		& $1.581\times10^{-5}$ & $9.651\times10^{-5}$ & $18.41877$\\
		320 & 801 & $0.00125$
		& $1.810\times10^{-5}$ & $1.022\times10^{-4}$ & $18.11249$\\
		\bottomrule
	\end{tabular}
\end{table}

The residuals and computed control costs remain small over the tested
meshes, although the residuals are not monotone. Since the number of MINRES
iterations is fixed, the table does not by itself provide a rigorous
discretization-convergence rate.

\subsection{Parametric sensitivity}
\label{subsec:sensitivity}

Finally, we examine the dependence of the moving-control computation
on the delay $h$ and on the decay rate $a$ in the exponential kernels $\mathfrak M(t)=e^{-at},
\widetilde{\mathfrak M}(t)=e^{-at}.$
Thus, as in the reference computation, the two kernels are taken equal
in the numerical experiments, while the analytical framework allows
them to be different.
We use $N=60$, $M=401$, and $80$ MINRES iterations.

\begin{table}[h!]
	\centering
	\caption{Sensitivity of the discrete HUM computation to the delay $h$
		and memory decay rate $a$. The first three rows vary $h$ with $a=1$;
		the last two vary $a$ with $h=0.10$.}
	\label{tab:sensitivity}
	\small
	\setlength{\tabcolsep}{3.5pt}
	\begin{tabular}{cccccc}
		\toprule
		$h$ & $a$ & $R_T$ & $R_M$ & $R_H$ & $\|u\|_{L^2(Q)}^2$\\
		\midrule
		0.05 & 1.0
		& $7.689\times10^{-5}$ & $1.255\times10^{-4}$
		& $1.337\times10^{-4}$ & 12.73246\\
		0.10 & 1.0
		& $8.567\times10^{-5}$ & $3.825\times10^{-4}$
		& $4.007\times10^{-4}$ & 16.80016\\
		0.15 & 1.0
		& $1.789\times10^{-4}$ & $1.019\times10^{-3}$
		& $1.022\times10^{-3}$ & 24.73726\\
		\midrule
		0.10 & 2.0
		& $7.735\times10^{-5}$ & $2.360\times10^{-4}$
		& $2.813\times10^{-4}$ & 11.85849\\
		0.10 & 0.5
		& $9.748\times10^{-5}$ & $5.156\times10^{-4}$
		& $4.880\times10^{-4}$ & 21.08000\\
		\bottomrule
	\end{tabular}
\end{table}

The sensitivity curves are displayed in
Figure~\ref{fig:sensitivity_plot}. Within the tested finite-dimensional
discretizations and fixed MINRES budget, increasing the delay produces
larger computed control cost and larger terminal residuals in the
tested cases, whereas faster memory decay produces a smaller computed
control cost. These observations are numerical trends and are not
asserted as monotonicity properties of the continuous control problem.

\begin{figure}[h!]
	\centering
	\includegraphics[width=0.50\textwidth]
	{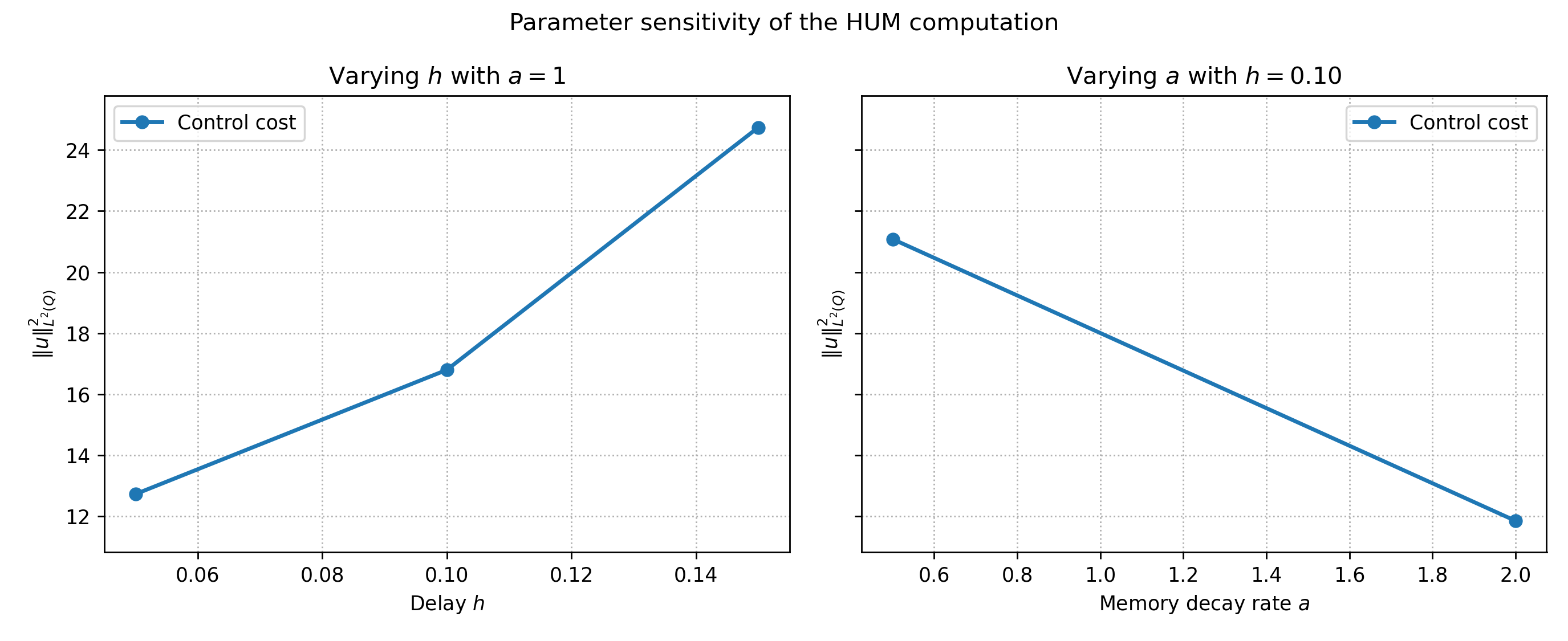}
	\caption{
	Sensitivity of the synthesized HUM control cost $\|u\|_{L^2(Q)}^2$ to system parameters under a moving control strategy. Left: Monotonic increase in control effort as the delay parameter $h$ increases ($a=1$). Right: Reduction in required control energy as the memory decay rate $a$ increases ($h=0.10$).
	}
	\label{fig:sensitivity_plot}
\end{figure}

\subsection{Concluding remarks}
\label{subsec:num-conclusion}

The numerical experiments provide a constructive implementation of the
HUM framework for the bounded-delay case $\mathfrak A_1=I$ with
exponential memory. In the numerical tests we take
$\widetilde{\mathfrak M}=\mathfrak M$, although the analytical
definition permits these kernels to differ. The discrete terminal map
simultaneously accounts for the terminal state, the terminal history
condition, and the full terminal memory constraint.

The numerical comparison shows substantially smaller residuals and
computed control cost for the moving configuration under the prescribed
discrete MINRES budget. The mesh-refinement and parameter-sensitivity
experiments further illustrate the behavior of the discrete
construction. These computations are numerical evidence for the
specified finite-dimensional implementations only and do not replace
the analytical controllability and observability results established in
the previous sections.

All numerical comparisons are performed for the bounded operator
$\mathfrak A_1=I$. Consequently, the experiments do not involve an
unbounded delayed operator and should not be interpreted as numerical
evidence for results outside the bounded-delay framework.

	\section{Conclusion}
	\label{sec:conclusion}
	
	This paper developed a controllability framework for parabolic evolution
	equations with simultaneous time-delay and memory effects. We introduced
	\emph{delay and memory-type null controllability} and established its dual
	characterization through an observability inequality for the
	adjoint system. For bounded delay operators and exponentially represented
	memory kernels, a global Carleman estimate was derived using a
	parabolic--ODE formulation, a temporal cutoff, and moving-control geometry.
	This yielded the desired observability inequality and, through the Hilbert
	Uniqueness Method, the corresponding controllability result. Numerical
	experiments further show that, for the tested configurations, moving controls
	achieve lower control costs and better terminal-state reduction than fixed
	controls. Future work will consider more general memory kernels, time-dependent
	delays, and unbounded delay operators.

	\section*{Author Contributions}
	\textbf{Dev Prakash Jha (First Author):} Conceptualization, Methodology, Software, Validation, Formal analysis, Investigation, Resources, Writing -- Original Draft, Writing -- Review \& Editing, Visualization.
	\textbf{Raju K. George (Second Author):} Supervision, Project administration.
	
	\section*{Acknowledgments}
	The authors sincerely thank the anonymous reviewers for their constructive comments and careful reading, which significantly improved the quality and presentation of this paper.

\bibliographystyle{unsrtnat} 
\bibliography{references} 

@article{allal2023null,
  title={Null controllability of coupled systems of degenerate parabolic integro-differential equation},
  author={Allal, Brahim and Fragnelli, Genni and Salhi, Jawad and others},
  journal={Electronic Journal of Differential Equations},
  volume={2023},
  number={18},
  pages={1--22},
  year={2023}
}

@article{wang2005approximate,
	title={Approximate controllability of delayed semilinear control systems},
	author={Wang, Lianwen},
	journal={International Journal of Stochastic Analysis},
	volume={2005},
	number={1},
	pages={67--76},
	year={2005},
	publisher={Wiley Online Library}
}

@article{wang2009approximate,
	title={Approximate controllability for integrodifferential equations with multiple delays},
	author={Wang, LW},
	journal={Journal of Optimization Theory and Applications},
	volume={143},
	number={1},
	pages={185--206},
	year={2009},
	publisher={Springer}
}

@book{curtain2012introduction,
	title={An introduction to infinite-dimensional linear systems theory},
	author={Curtain, Ruth F and Zwart, Hans},
	year={2012},
	publisher={Springer Science \& Business Media}
}

@article{barbu2000controllability,
  author={Barbu, V. and Iannelli, M.},
  title={Controllability of the heat equation with memory},
  journal={Differential and Integral Equations},
  volume={13},
  number={10–12},
  pages={1393--1412},
  year={2000},
  doi={10.57262/die/1356061131}
}

@article{cattaneo1958form,
  title={A form of heat-conduction equations which eliminates the paradox of instantaneous propagation},
  author={Cattaneo, Carlo},
  journal={Comptes rendus},
  volume={247},
  pages={431},
  year={1958}
}

@article{casas2018measure,
  title={Measure control of a semilinear parabolic equation with a nonlocal time delay},
  author={Casas, Eduardo and Mateos, Mariano and Tr{"o}ltzsch, Fredi},
  journal={SIAM Journal on Control and Optimization},
  volume={56},
  number={6},
  pages={4434--4460},
  year={2018},
  publisher={SIAM}
}

@article{chaves2014null,
  title={Null controllability of a system of viscoelasticity with a moving control},
  author={Chaves-Silva, Felipe W and Rosier, Lionel and Zuazua, Enrique},
  journal={Journal de Math{\'e}matiques Pures et Appliqu{\'e}es},
  volume={101},
  number={2},
  pages={198--222},
  year={2014},
  publisher={Elsevier}
}

@article{chaves2017controllability,
  title={Controllability of evolution equations with memory},
  author={Chaves-Silva, Felipe W and Zhang, Xu and Zuazua, Enrique},
  journal={SIAM Journal on Control and Optimization},
  volume={55},
  number={4},
  pages={2437--2459},
  year={2017},
  publisher={SIAM}
}

@book{coron2007control,
  title={Control and nonlinearity},
  author={Coron, Jean-Michel},
  series={Grundlehren der mathematischen Wissenschaften},
  number={136},
  year={2007},
  publisher={American Mathematical Soc.}
}

@book{fursikov1996controllability,
  title={Controllability of evolution equations},
  author={Fursikov, Andrej Vladimirovi{\v{c}} and Imanuvilov, O. Yu},
  journal={Seoul National University},
  year={1996},
  publisher={The University of Michigan}
}

@article{gallegos2024relaxed,
  title={Relaxed excitation conditions for robust identification and adaptive control using estimation with memory},
  author={Gallegos, Javier and Aguila-Camacho, Norelys},
  journal={SIAM Journal on Control and Optimization},
  volume={62},
  number={1},
  pages={1--21},
  year={2024},
  publisher={SIAM}
}

@article{guerrero2013remarks,
  title={Remarks on non controllability of the heat equation with memory},
  author={Guerrero, Sergio and Imanuvilov, Oleg Yurievich},
  journal={ESAIM: Control, Optimization and Calculus of Variations},
  volume={19},
  number={1},
  pages={288--300},
  year={2013},
  publisher={EDP Sciences}
}

@article{gurtin1968general,
  title={A general theory of heat conduction with finite wave speeds},
  author={Gurtin, Morton E and Pipkin, Allen C},
  journal={Archive for Rational Mechanics and Analysis},
  volume={31},
  number={2},
  pages={113--126},
  year={1968},
  publisher={Springer}
}

@article{halanay2012lack,
  title={Lack of controllability of the heat equation with memory},
  author={Halanay, Andrei and Pandolfi, Luciano},
  journal={Systems \& Control Letters},
  volume={61},
  number={10},
  pages={999--1002},
  year={2012},
  publisher={Elsevier}
}

@inproceedings{kalman1960general,
  title={On the general theory of control systems},
  author={Kalman, Rudolf E},
  booktitle={Proceedings first international conference on automatic control, Moscow, USSR},
  pages={481--492},
  year={1960}
}

@article{khapalov2001mobile,
  title={Mobile point controls versus locally distributed ones for the controllability of the semilinear parabolic equation},
  author={Khapalov, Alexander},
  journal={SIAM journal on control and optimization},
  volume={40},
  number={1},
  pages={231--252},
  year={2001},
  publisher={SIAM}
}

@article{khodja2014null,
  title={Null controllability of retarded parabolic equations},
  author={Khodja, Farid Ammar and Bouzidi, Cherif and Dupaix, C{\'e}dric and Maniar, Lahcen},
  journal={Math. Control Relat. Fields},
  volume={4},
  number={1},
  pages={1--15},
  year={2014}
}

@article{kumar2022approximate,
  title={Approximate controllability of linear parabolic equation with memory},
  author={Kumar, Anil and Pani, Amiya K and Joshi, Mohan C},
  journal={Computers \& Mathematics with Applications},
  volume={128},
  pages={320--330},
  year={2022},
  publisher={Elsevier}
}

@article{liu2025memory,
  title={Memory-Type Acoustic Boundary Stabilization of a Variable-Coefficient Wave Equation with Time-Varying Delay},
  author={Liu, Yu-Xiang and Zhang, Lei},
  journal={SIAM Journal on Control and Optimization},
  volume={63},
  number={5},
  pages={3807--3823},
  year={2025},
  publisher={SIAM}
}

@article{lions1988controlabilite,
  title={Contr{\^o}labilit{\'e} exacte},
  author={Lions, Jacques-Louis},
  journal={Perturbations et stabilisation de systemes distribues},
  year={1988},
  publisher={Masson}
}

@article{lu2017null,
  title={Null controllability for wave equations with memory},
  author={L{\"u}, Qi and Zhang, Xu and Zuazua, Enrique},
  journal={Journal de Math{\'e}matiques Pures et Appliqu{\'e}es},
  volume={108},
  number={4},
  pages={500--531},
  year={2017},
  publisher={Elsevier}
}

@article{ma2025convergence,
  title={Convergence of minimal norm control problems of linear heat equations with time delay},
  author={Ma, Yaxing and Wang, Lijuan and Zhou, Xiuxiang},
  journal={Mathematical Control and Related Fields},
  volume={15},
  number={3},
  pages={1112--1120},
  year={2025},
  publisher={Mathematical Control and Related Fields}
}

@article{martin2013null,
  title={Null controllability of the structurally damped wave equation with moving control},
  author={Martin, Philippe and Rosier, Lionel and Rouchon, Pierre},
  journal={SIAM Journal on Control and Optimization},
  volume={51},
  number={1},
  pages={660--684},
  year={2013},
  publisher={SIAM}
}

@article{pandolfi2013boundary,
  title={Boundary controllability and source reconstruction in a viscoelastic string under external traction},
  author={Pandolfi, Luciano},
  journal={Journal of Mathematical Analysis and Applications},
  volume={407},
  number={2},
  pages={464--479},
  year={2013},
  publisher={Elsevier}
}

@article{russell1978controllability,
  title={Controllability and stabilizability theory for linear partial differential equations: recent progress and open questions},
  author={Russell, David L},
  journal={SIAM Review},
  volume={20},
  number={4},
  pages={639--739},
  year={1978},
  publisher={SIAM}
}

@article{wang2024observability,
  title={Observability for heat equations with time-dependent analytic memory},
  author={Wang, Gengsheng and Zhang, Yubiao and Zuazua, Enrique},
  journal={Archive for Rational Mechanics and Analysis},
  volume={248},
  number={6},
  pages={115},
  year={2024},
  publisher={Springer}
}

@article{wang2022flow,
  title={Flow decomposition for heat equations with memory},
  author={Wang, Gengsheng and Zhang, Yubiao and Zuazua, Enrique},
  journal={Journal de Math{\'e}matiques Pures et Appliqu{\'e}es},
  volume={158},
  pages={183--215},
  year={2022},
  publisher={Elsevier}
}

@article{zhou2014interior,
  title={Interior approximate and null controllability of the heat equation with memory},
  author={Zhou, Xiuxiang and Gao, Hang},
  journal={Computers \& Mathematics with Applications},
  volume={67},
  number={3},
  pages={602--613},
  year={2014},
  publisher={Elsevier}
}

@article{zhou2023perturbations,
  title={Perturbations of time optimal control problems for parabolic integro-differential equations},
  author={Zhou, Xiuxiang},
  journal={SIAM Journal on Control and Optimization},
  volume={61},
  number={4},
  pages={2069--2087},
  year={2023},
  publisher={SIAM}
}

@article{zhao2025null,
  title={Null Controllability for Stochastic Semidiscrete Parabolic Equations},
  author={Zhao, Qingmei},
  journal={SIAM Journal on Control and Optimization},
  volume={63},
  number={3},
  pages={2007--2028},
  year={2025},
  publisher={SIAM}
}

@incollection{zuazua2007controllability,
  title={Controllability and observability of partial differential equations: some results and open problems},
  author={Zuazua, Enrique},
  booktitle={Handbook of differential equations: evolutionary equations},
  volume={3},
  pages={527--621},
  year={2007},
  publisher={Elsevier}
}

\end{document}